\documentclass[leqno]{article}
\usepackage{graphicx} 
\usepackage{float}
\usepackage[T1]{fontenc}
\usepackage[margin=1in]{geometry}
\usepackage{amsmath,amsthm,amssymb,amsfonts}
\usepackage{hyperref}
\usepackage{xcolor}
\usepackage{csquotes}
\usepackage{mathtools}

\usepackage{algorithm,algorithmic}
\usepackage{mystylefile}

\usepackage{upgreek,comment,url}
\newtheorem{theorem}{Theorem}
\newtheorem{lemma}{Lemma}

\graphicspath{ {Figures/} }
\newcommand{\email}[1]{\protect\href{mailto:#1}{#1}}
\usepackage[
backend=biber,
style=alphabetic,
]{biblatex}

\title{\textbf{Small singular values can increase \\ in lower precision}}
\author{Christos Boutsikas\\ Purdue University \\ \email{cboutsik@purdue.edu} \and 
Petros Drineas \\ Purdue University \\ \email{pdrineas@purdue.edu} \and 
Ilse C.F. Ipsen \\ North Carolina State University \\ \email{ipsen@ncsu.edu} }
\date{}

\begin{document}

\maketitle

\begin{abstract}
We perturb a real matrix $A$ of full column rank, and derive 
lower bounds for the smallest singular values of the perturbed matrix,
in terms of normwise absolute perturbations. Our bounds, which extend existing lower-order expressions, demonstrate the potential increase in 
the smallest 
singular values, and represent a qualitative model
for the increase in the small singular values after a matrix has been
downcast to a lower arithmetic precision.
Numerical experiments confirm the qualitative validity of this model
and its ability to predict singular values changes in the presence of 
decreased arithmetic precision.
\end{abstract}

\section{Introduction}
Given a real, full column-rank matrix $\ma$, we present 
lower bounds for the  smallest singular values of a perturbed matrix $\ma+\me$.

\subsection{Motivation}
We investigate the change in the computed singular values of a tall and skinny 
matrix $\ma\in\rmn$ with $m\geq n$ and $\rank(\ma)=n$ when $\ma$ is demoted, that is,
downcast to a lower arithmetic precision.

We have observed that demotion to lower precision 
can improve the conditioning of
$\ma$ by significantly
increasing the computed small singular values, while leaving large singular values mostly unharmed. 
For instance, if  the smallest singular value of~$\ma$ is on the order of double precision
roundoff,
\begin{equation*}
\smin\left(\mathtt{double}(\ma)\right)\approx 10^{-16},
\end{equation*}
then downcasting $\ma$ to single precision can increase 
the smallest singular value to single precision roundoff,
\begin{equation*}
 \smin\left(\mathtt{double}(\mathtt{single}(\ma))\right)\approx 10^{-8}.
 \end{equation*}
 This phenomenon has been observed before, as the following quotes illustrate:
\begin{quote}
\textit{$\ldots$ small singular values tend to increase \cite[page 266]{SSun90}}\\
\textit{$\ldots$ even an approximate inverse of an arbitrarily ill-conditioned\\
 matrix does,
in general,  contain useful information \cite[page 260]{Rump09}}\\
\textit{This is due to a kind of regularization by rounding to working precision
\cite[page 261]{Rump09}}
\end{quote} \vspace*{0.2cm}

\subsection{Modelling demotion to lower precision in terms of perturbations}
We model the downcasting of a matrix to lower precision in terms of
normwise absolute perturbations.

The accumulated error from typical singular value algorithms in Matlab and Julia 
is a normwise absolute error \cite[section 8.6]{GovL13}.
Thus, if we downcast a matrix $\ma\in\rmn$ with $m\geq n$ to single precision
and compute the singular values of the demoted matrix,
the resulting error can be represented as an absolute perturbation~$\me$.
According to Weyl's inequality \cite[Corollary 8.6.2]{GovL13},
corresponding singular 
values\footnote{The singular values of a matrix $\ma\in\rmn$ are labelled in 
non-increasing order, by  $\sigma_1(\ma)\geq \sigma_2(\ma)\geq \ldots\geq 
\sigma_{\min\{m,n\}}(\ma)$.}
of $\ma$ change by at most $\|\me\|_2$,
\begin{align*}
|\sigma_j(\ma+\me)-\sigma_j(\ma)|\leq \|\me\|_2\approx 10^{-8},
\qquad 1\leq j\leq n.
\end{align*}
The bound implies that singular values larger than single precision roundoff,
i.e. $\sigma_j(\ma)\gg\|\me\|_2$, remain essentially the same,
\begin{equation*}
\sigma_j(\ma)\approx\sigma_j(\ma)-\underbrace{\|\me\|_2}_{10^{-8}}
\leq \sigma_j(\ma+\me)\leq \sigma_j(\ma)+\underbrace{\|\me\|_2}_{10^{-8}}
\approx \sigma_j(\ma),
\end{equation*}
while it is inconclusive about small singular values on the order of double precision
roundoff,
  \begin{equation*}
\underbrace{\sigma_{\ell}(\ma)}_{10^{-16}}-\underbrace{\|\me\|_2}_{10^{-8}}
\leq \sigma_{\ell}(\ma+\me)\leq  \underbrace{\sigma_{\ell}(\ma)}_{10^{-16}}+
\underbrace{\|\me\|_2}_{10^{-8}}.
\end{equation*}

\subsection{Our contributions}
Our main results are normwise absolute lower bounds on the smallest singular value cluster of a perturbed matrix. The bounds, compactly summarized in Theorem~\ref{t_i1} below, testify to a definitive increase in the perturbed small singular values. The qualitative validity of the bounds
is confirmed by the numerical experiments in section~\ref{s_exp}.
Our assumptions are not restrictive and merely
require the smallest singular value cluster to be separated by a small gap from
the remaining singular values.

Theorem~\ref{t_i1} improves the second-order perturbation expansions in
 \cite{Ste84,SSun90,SteM06}, because it is a  true lower bound
and, unlike \cite{Ste84}, \cite[Lemma V.4.5]{SSun90}, it needs no assumptions on 
the size of  the offdiagonal blocks.
Its proof (Sections~\ref{s_single} and~\ref{s_cluster}) follows from Theorem~\ref{t_det4} and a restatement of Assumptions~\ref{ass_2}. The special case where $r=1$ reduces to Assumptions~\ref{ass_1} and Theorem~\ref{t_det3}.

\begin{theorem}\label{t_i1}
Let $\ma\in\rmn$ with $m\geq n$ have $\rank(\ma)\geq n-r$ for some $r\geq 1$.
Let $\ma=\mU\msig\mv^T$ be a full singular value decomposition,
where $\msig\in\rmn$ is diagonal, and
$\mU\in\rmm$ and $\mv\in\rnn$ are orthogonal matrices. 
Partition commensurately,
\begin{align*}
\msig=\begin{bmatrix} \msig_1 & \vzero \\ \vzero & \msig_2\\ \vzero & \vzero \end{bmatrix},\qquad 
\mU^T\me\mv=\begin{bmatrix} \me_{11} & \me_{12} \\ \me_{21} & \me_{22}\\ \me_{31}& \me_{32} \end{bmatrix},
\end{align*}
where $\msig_1,\me_{11}\in\real^{(n-r)\times (n-r)}$ with $\msig_1$ nonsingular diagonal;
and $\msig_2,\me_{22}\in\real^{r\times r}$ with $\msig_2$ diagonal.

If $1/\|\msig_{1}^{-1}\|_2>4\|\me\|_2$ and $\|\msig_2\|_2<\|\me\|_2$,
then\footnote{The eigenvalues of a symmetric matrix $\mh\in\real^{k\times k}$ are labelled in non-increasing order, by $\lambda_1(\mh)\geq \cdots \geq\lambda_k(\mh)$.}
\begin{align*}
\sigma_{n-r+j}(\ma+\me)^2\geq\lambda_j(\me_{32}^T\me_{32}+(\msig_2+\me_{22})^T(\msig_2+\me_{22}) -\mr_3)-r_4,\quad 1\leq j\leq r,
\end{align*}
where $\mr_3$ contains terms of order 3,
\begin{align*}
\mr_3\equiv \me_{12}^T\mw+\mw^T\me_{12},\qquad 
\mw\equiv(\msig_1+\me_{11})^{-T}\begin{bmatrix}\me_{21}^T& \me_{31}^T\end{bmatrix}
\begin{bmatrix} \msig_2+\me_{22}\\ \me_{32}\end{bmatrix}
\end{align*}
and $r_4$ contains terms of order 4 and higher,
\begin{align*}
r_4\equiv \|\mw\|_2^2
+4\frac{\|\me\|_2^2\,\|(\msig_1+\me_{11})^{-1}(\me_{12}+\mw)\|_2^2}
{1-4\|\me\|_2^2 \|(\msig_1+\me_{11})^{-1}\|_2^2}.
\end{align*}
\end{theorem}

\noindent Future work, sketched in section~\ref{s_future}, will refine the
above results towards a quantitative analysis that predicts the order of
magnitude of the increase, and the influential matrix properties,
in particular, the role of the singular value gap. We also note that our theorem does not directly account for computational precision. However, our experiments in Section~\ref{s_exp} demonstrate that when performing calculations in double precision, the “exact” singular values exhibit complete overlap with those computed in double precision. This suggests that perturbations arising from computational precision should not affect our results.

\subsection{Existing work}\label{s_lit}
There are many bounds for the smallest singular value of
general, unstructured matrices. The bounds for nonsingular matrices
$\ma\in\cnn$ in \cite{LinXie21,Shun22,Zou12} involve the factor 
$|\det(\ma)|^2\left(\frac{n-1}{\|\ma\|_F^2}\right)^{n-1}$, while
the ones in \cite{HongPan92,YuGu97} contain factors
like $|\det(\ma)|^2\left(\frac{n-1}{n}\right)^{(n-1)/2}$ and row and column norms.
The Schur complement-based bounds for strictly diagonally dominant matrices
in \cite{Huang08,Li20,Oishi23,Sang21,Varah75}
  depend on the degree of
 diagonal dominance, as do the
Gerschgorin type bounds for rectangular matrices in
\cite{Johnson89,JohnsonSzulc98}.

In contrast, we are bounding the smallest singular values of  
\textit{perturbed} matrices. The expressions for small singular values in
 \cite[Theorem]{Ste84},  \cite[Theorem 8]{SteM06},  \cite[Section V.4.2]{SSun90}
are second-order perturbation expansions rather than bounds, and 
require assumptions on the singular vectors. 

\subsection{Overview}
Our deterministic lower bounds for small singular values
of $\ma+\me$ are based on eigenvalue bounds
for $(\ma+\me)^T(\ma+\me)$.
We present normwise absolute bounds for a single smallest singular value (section~\ref{s_single}) and
for a cluster of small singular values (section~\ref{s_cluster}).
The numerical experiments (section~\ref{s_exp}) confirm
the qualitative increase in small singular values resulting from the demotion 
of the matrix to lower precision.
A brief discussion of future work (section~\ref{s_future}) concludes the paper.

\section{A single smallest singular value}\label{s_single}
We perturb a matrix that has a single smallest singular value,
 and derive a lower bound for the smallest singular value of the perturbed matrix
 in terms of normwise absolute perturbations
 (Section~\ref{s_slb}),  based on eigenvalue bounds (Section~\ref{s_sev}).

\subsection{Auxiliary eigenvalue results}\label{s_sev}
We square the singular values of $\ma\in\rmn$ and consider instead the eigenvalues 
of the symmetric positive semi-definite matrix $\mb\equiv \ma^T\ma\in\rnn$.

For a symmetric positive semi-definite matrix $\mb\in\rnn$ 
with a single smallest eigenvalue $\lmin(\mb)$, we present two expressions for $\lmin(\mb)$
with different assumptions
(Lemmas \ref{l_dp1} and~\ref{r_dp1}), and two lower bounds in terms
of normwise absolute perturbations (Theorems \ref{t_dp4} 
and~\ref{t_dp5}).

We assume that $\lmin(\mb)$ is separated 
from the remaining eigenvalues, in the sense that it is strictly smaller than
the smallest eigenvalue of the leading principal submatrix $\mb_{11}$
of order $n-1$. The equality  below expresses $\lmin(\mb)$ in terms of itself.

\begin{lemma}[Exact expression]\label{l_dp1}
Let $\mb\in\rnn$ be symmetric positive semi-definite with $\rank(\mb)\geq n-1$,
and partition
\begin{align*}
\mb=\begin{bmatrix}\mb_{11} & \vb \\ \vb^T & \beta\end{bmatrix} \qquad \text{where}\quad  \mb_{11}\in\real^{(n-1)\times (n-1)}.
\end{align*}
Then
\begin{align}\label{e_dp1}
0\leq\lmin(\mb)\leq \beta.
\end{align}
If also  $\lmin(\mb)<\lmin(\mb_{11})$ then 
\begin{align*}
\lmin(\mb)=\beta-\vb^T(\mb_{11}-\lmin(\mb)\,\mi)^{-1}\vb,
\end{align*}
\end{lemma}

\begin{proof}
Abbreviate $\tlmin\equiv\lmin(\mb)$. 
 The positive semi-definiteness of $\mb$ implies the lower bound in (\ref{e_dp1}),
 while the variational inequalities imply the upper bound, 
\begin{align*}
0\leq \tlmin=\min_{\|\vx\|_2=1}{\vx^T\mb\vx}\leq \ve_n^T\mb\ve_n=\beta.
\end{align*}
To show the expression for $\tlmin$, observe that the shifted matrix
\begin{align*}
\mb-\tlmin \mi=\begin{bmatrix} \mb_{11}-\tlmin \mi& \vb\\ \vb^T & \beta-\tlmin
\end{bmatrix} 
\end{align*}
is singular.
From  the assumption $\tlmin<\lmin(\mb_{11})$ follows that $\mb_{11}-\tlmin\mi$ is nonsingular.
So we can determine the block LU decomposition
$\mb-\tlmin\mi=\ml\widehat{\mU}$ with
\begin{align*}
\ml&\equiv\begin{bmatrix}\mi & \vzero\\
\vb^T(\mb_{11}-\tlmin\mi)^{-1} & 1\end{bmatrix}, \\
\widehat{\mU}&\equiv \begin{bmatrix}\mb_{11}-\tlmin \mi& \vb\\
\vzero & \beta-\tlmin -\vb^T(\mb_{11}-\tlmin\mi)^{-1}\vb\end{bmatrix}.
\end{align*}
Since $\mb-\tlmin \mi$ is singular and the unit triangular matrix $\ml$ is 
nonsingular, the block upper triangular matrix $\widehat{\mU}$ has no choice but to be singular.
Its leading principal submatrix $\mb_{11}-\tlmin \mi$ is nonsingular by assumption, which
leaves the (2,2) element to be singular, but it being a scalar implies
\begin{align*}
 \beta-\tlmin -\vb^T(\mb_{11}-\tlmin\mi)^{-1}\vb=0.
 \end{align*}
This gives the expression for $\tlmin$.
 \end{proof}

If $\vb=\vzero$ then Lemma~\ref{l_dp1} correctly asserts that 
$\lmin(\mb)=\beta$. 

Lemma~\ref{r_dp1} below presents the same
expression for $\lmin(\mb)$ as in Lemma~\ref{l_dp1}, but under a stronger albeit more useful
assumption.

\begin{lemma}[Exact expression with stronger assumption]\label{r_dp1}
Let $\mb\in\rnn$ be symmetric positive semi-definite with $\rank(\mb)\geq n-1$, and partition
\begin{align*}
\mb=\begin{bmatrix}\mb_{11} & \vb \\ \vb^T & \beta\end{bmatrix} \qquad \text{where}\quad  \mb_{11}\in\real^{(n-1)\times (n-1)}.
\end{align*}
If  $\beta<\lmin(\mb_{11})$ then 
\begin{align*}
\lmin(\mb)=\beta-\vb^T(\mb_{11}-\lmin(\mb)\,\mi)^{-1}\vb\geq 0.
\end{align*}
\end{lemma}

\begin{proof}
The upper bound (\ref{e_dp1}) combined with  the assumption   $\beta<\lmin(\mb_{11})$ 
implies the assumption in Lemma~\ref{l_dp1},
\begin{align}\label{e_r21}
0\leq \lmin(\mb)\leq \beta<\lmin(\mb_{11}).
\end{align}
\end{proof}

The subsequent lower bounds for $\lmin(\mb)$ are informative if the offdiagonal
part has small norm.

\begin{theorem}[First lower bound]\label{t_dp4}
Let $\mb\in\rnn$ be symmetric positive semi-definite with $\rank(\mb)\geq n-1$, and partition
\begin{align*}
\mb=\begin{bmatrix}\mb_{11} & \vb \\ \vb^T & \beta\end{bmatrix} \qquad \text{where}\quad  \mb_{11}\in\real^{(n-1)\times (n-1)}.
\end{align*}
If  $\beta<\lmin(\mb_{11})$ then 
\begin{align*}
\lmin(\mb)\geq \beta-\vb^T\mb_{11}^{-1}\vb-\frac{\beta\,\|\mb_{11}^{-1}\vb\|_2^2}{1-\beta\,\|\mb_{11}^{-1}\|_2}.
\end{align*}
\end{theorem}

\begin{proof}
Abbreviate $\tlmin\equiv\lmin(\mb)$.
From (\ref{e_r21}) follows that $\mb_{11}$ and 
$\mb_{11}-\tlmin\mi$ are nonsingular. Combined
with the symmetric positive semi-definiteness of $\mb_{11}$
this gives
\begin{align*}
\tlmin<\lmin(\mb_{11})=1/\|\mb_{11}^{-1}\|_2,
\end{align*}
hence
\begin{align}\label{e_t22a}
\|\tlmin\mb_{11}^{-1}\|_2<1.
\end{align}
Thus we can apply the Sherman-Morrison formula \cite[Section 2.1.4]{GovL13},
\begin{equation*}
(\mb_{11}-\tlmin\mi)^{-1}=\mb_{11}^{-1}+\tlmin\,\mb_{11}^{-1}(\mi-\tlmin\,\mb_{11}^{-1})^{-1}\mb_{11}^{-1},
\end{equation*}
and substitute the above
into the expression for $\tlmin$ from Lemma~\ref{l_dp1}, 
\begin{align}\label{e_t22c}
\tlmin&=\beta-\vb^T\mb_{11}^{-1}\vb-\tlmin\,\vb^T\mb_{11}^{-1}(\mi-\tlmin\,\mb_{11}^{-1})^{-1}\mb_{11}^{-1}\vb.
\end{align}
The symmetric positive semi-definiteness of $\mb$ implies
that $\beta\geq 0$ and $\vb^T\mb_{11}^{-1}\vb\geq 0$, hence it remains
to bound the norm of the remaining summand.
From the symmetry of $\mb_{11}$ and the invariance of the two-norm under transposition follows
\begin{align}
\|\vb^T\mb_{11}^{-1}(\mi-\tlmin\,\mb_{11}^{-1})^{-1}\mb_{11}^{-1}\vb\|_2
&\leq \|\vb^T\mb_{11}^{-1}\|_2\|(\mi-\tlmin\,\mb_{11}^{-1})^{-1}\|_2
\|\mb_{11}^{-1}\vb\|_2\notag\\
&\leq \|\mb_{11}^{-1}\vb\|_2^2\|(\mi-\tlmin\,
\mb_{11}^{-1})^{-1}\|_2.\label{e_t22b}
\end{align}
The inequality (\ref{e_t22a}) allows us to apply the Banach lemma \cite[Lemma 2.3.3]{GovL13} to bound the norm of the inverse by
\begin{align*}
\|(\mi-\tlmin\,\mb_{11}^{-1})^{-1}\|_2\leq 
\frac{1}{1-\|\tlmin\mb_{11}^{-1}\|_2}
=\frac{1}{1-\tlmin\|\mb_{11}^{-1}\|_2.}
\end{align*}
Substitute this into (\ref{e_t22b}) and the resulting bound
into the expression for $\tlmin$ in (\ref{e_t22c}),
\begin{align*}
\tlmin\geq \beta-\vb^T\mb_{11}^{-1}\vb-\frac{\tlmin\,\|\mb_{11}^{-1}\vb\|_2^2}{1-\tlmin\,\|\mb_{11}^{-1}\|_2},
\end{align*}
and at last apply the upper bound~(\ref{e_dp1}).
\end{proof}

The lower bound in Theorem~\ref{t_dp4} is positive if $\|\vb\|_2$ is sufficiently 
small, in which case  $\lmin(\mb)\geq \beta -\mathcal{O}(\|\vb\|_2^2)$.
If $\vb=\vzero$ then (\ref{e_dp1}) 
and Theorem~\ref{t_dp4} imply $\lmin(\mb)=\beta$.

The slightly weaker bound below focusses on a `dominant part' 
of~$\mb_{11}$.

\begin{theorem}[Second lower bound]\label{t_dp5}
Let $\mb\in\rnn$ be symmetric positive semi-definite with $\rank(\mb)\geq n-1$, and partition
\begin{align*}
\mb=\begin{bmatrix}\mb_{11} & \vb \\ \vb^T & \beta\end{bmatrix} \qquad \text{where}\quad  \mb_{11}\in\real^{(n-1)\times (n-1)}.
\end{align*}
If  $\mb_{11}=\mc_{11}+\mc_{12}$ where
$\mc_{11}$ is symmetric positive definite with $\lmin(\mc_{11})>\beta$, 
and $\mc_{12}$ is symmetric positive semi-definite then
\begin{align*}
\lmin(\mb)\geq \beta-\vb^T\mc_{11}^{-1}\vb-\frac{\beta\,\|\mc_{11}^{-1}\vb\|_2^2}{1-\beta\,\|\mc_{11}^{-1}\|_2}.
\end{align*}
\end{theorem}

\begin{proof}
 Abbreviate $\tlmin\equiv\lmin(\mb)$.
From~(\ref{e_dp1}) and the assumption follows 
$\tlmin\leq\beta<\lmin(\mc_{11})$, hence $\mc_{11}$ and
$\mc_{11}-\tlmin\mi$ are nonsingular. Write
\begin{align*}
\mb_{11}-\tlmin\mi=\underbrace{\mc_{11}-\tlmin\mi}_{\mg}+\mc_{12}=
\mg^{1/2}\,(\mi+\underbrace{\mg^{-1/2}\,\mc_{12}\,\mg^{-1/2}}_{\mh})\,\mg^{1/2},
\end{align*}
where $\mg$ is symmetric positive definite, $\mg^{1/2}$ is its symmetric positive definite square root,
and $\mh$ is symmetric positive semi-definite.
The Loewner partial ordering\footnote{For Hermitian matrices $\ma$ and $\mb$,
$\ma\preceq \mb$ means that $\ma-\mb$ is positive semi-definite.}
implies $\mi\preceq\mi+\mh$.
From \cite[Corollary 7.7.4]{HoJoI} follows $(\mi+\mh)^{-1}\preceq\mi^{-1}=\mi$. 
Thus
\begin{align*}
(\mb_{11}-\tlmin\mi)^{-1}&=\mg^{-1/2}(\mi+\mh)^{-1}\mg^{-1/2}\\
&\preceq \mg^{-1/2}\mi\mg^{-1/2}=\mg^{-1}=(\mc_{11}-\tlmin\mi)^{-1}.
\end{align*}
Substituting $(\mb_{11}-\tlmin\mi)^{-1}\preceq (\mc_{11}-\tlmin\mi)^{-1}$
into the expression for $\tlmin$ in Lemma~\ref{l_dp1} gives
\begin{align*}
\tlmin\geq \beta-\vb^T(\mc_{11}-\tlmin\mi)^{-1}\vb.
\end{align*}
We continue as in the proof of Theorem~\ref{t_dp4} with the
Sherman-Morrison formula \cite[Section 2.1.4]{GovL13},
\begin{align*}
\tlmin&\geq\beta-\vb^T\mc_{11}^{-1}\vb-\tlmin\,\vb^T\mc_{11}^{-1}(\mi-\tlmin\,\mc_{11}^{-1})^{-1}\mc_{11}^{-1}\vb\\
&\geq \beta-\vb^T\mc_{11}^{-1}\vb-\frac{\tlmin\,\|\mc_{11}^{-1}\vb\|_2^2}{1-\tlmin\,\|\mc_{11}^{-1}\|_2},
\end{align*}
and at last apply (\ref{e_dp1}).
 \end{proof}

If $\mc_{12}=\vzero$, then Theorem~\ref{t_dp5} reduces to Theorem~\ref{t_dp4}.

\subsection{A lower bound for the smallest singular value}\label{s_slb}
We consider a matrix with a distinct smallest singular value.
Based on the eigenvalue bounds in section~\ref{s_sev}, we derive a lower bound
for the smallest singular value of a perturbed matrix (Theorem~\ref{t_det3}) in terms
of normwise absolute perturbations.
We start with a summary of all assumptions (Assumptions~\ref{ass_1}), and end with a 
discussion of their generality (Remark~\ref{r_det3}).

\begin{assumptions}\label{ass_1}
Let $\ma\in\rmn$ with $m\geq n$ have $\rank(\ma)\geq n-1$.
Let $\ma=\mU\msig\mv^T$ be a full singular value decomposition,
where $\msig\in\rmn$ is diagonal, and
$\mU\in\rmm$ and $\mv\in\rnn$ are orthogonal matrices. Partition commensurately,
\begin{align*}
\msig=\begin{bmatrix} \msig_1 & \vzero \\ \vzero & \smin\\ \vzero & \vzero \end{bmatrix},\qquad 
\me=\mU\begin{bmatrix} \me_{11} & \ve_{12} \\ \ve_{21}^T & e_{22}\\ \me_{31}& \ve_{32} \end{bmatrix}\mv^T,
\end{align*}
where $\msig_1\in\real^{(n-1)\times (n-1)}$ is nonsingular diagonal, 
and $\smin\geq 0$.
\end{assumptions}

For a matrix with a single smallest singular value, we corroborate the observation that 
`small singular values 
tend to increase'  \cite[page 266]{SSun90}. Motivated by the 
second-order expressions in
terms of absolute perturbations \cite[Section V.4.2]{SSun90} and \cite[Theorem~8]{SteM06}, 
we derive a true lower bound.

\begin{theorem}\label{t_det3}
Let $\ma,\me\in\rmn$ satisfy Assumptions~\ref{ass_1}.
If $1/\|\msig_{1}^{-1}\|_2>4\|\me\|_2$ and $\smin<\|\me\|_2$,
then 
\begin{align*}
\smin(\ma+\me)^2\geq\|\ve_{32}\|_2^2+(\smin+e_{22})^2 -r_3-r_4,
\end{align*}
where $r_3$ contains terms of order 3,
\begin{align*}
r_3\equiv 2\ve_{12}^T\vw\qquad 
\vw\equiv(\msig_1+\me_{11})^{-T}\begin{bmatrix}\ve_{21}& \me_{31}^T\end{bmatrix}
\begin{bmatrix} e_{22}+\smin\\ \ve_{32}\end{bmatrix},
\end{align*}
and $r_4$ contains terms of order 4 and higher,
\begin{align*}
r_4\equiv \|\vw\|_2^2
+4\frac{\|\me\|_2^2\,\|(\msig_1+\me_{11})^{-1}(\ve_{12}+\vw)\|_2^2}
{1-4\|\me\|_2^2\|(\msig_1+\me_{11})^{-1}\|_2^2}.
\end{align*}
\end{theorem}

\begin{proof}
We square the singular values of $\ma+\me$, and consider the eigenvalues of
\begin{align*}
\mb\equiv (\ma+\me)^T(\ma+\me)=
\mv\begin{bmatrix} \mb_{11} & \vb \\ \vb^T & \beta\end{bmatrix}\mv^T
\end{align*}
where
\begin{align}
\mb_{11}&=\underbrace{(\msig_1+\me_{11})^T(\msig_1+\me_{11})}_{\mc_{11}}
+\underbrace{\ve_{21}\ve_{21}^T+\me_{31}^T\me_{31}}_{\mc_{12}}\label{e_C11}\\
\begin{split} \label{e_aux7}
\beta&=\|\ve_{12}\|_2^2+(\smin+e_{22})^2+\|\ve_{32}\|_2^2\\
\vb& = (\msig_1+\me_{11})^T\ve_{12}+\ve_{21}(\smin+e_{22}) +\me_{31}^T\ve_{32}.
\end{split}\end{align}
From $\smin(\msig_1)> 4\|\me\|_2$ follows that
$\mc_{11}$ is  symmetric positive definite, while $\mc_{12}$ is symmetric positive semi-definite and contains only second order terms.
Abbreviate $\tlmin\equiv \lmin(\mb)=\smin(\ma+\me)^2$.

The proof proceeds in two steps:
\begin{enumerate}
\item Confirming that  $\mc_{11}$ satisfies the assumptions of Theorem~\ref{t_dp5}.
\item Deriving the lower bound for $\tlmin$ from Theorem~\ref{t_dp5}.
\end{enumerate}
\smallskip

\paragraph{1. Confirm that  $\mc_{11}$ satisfies the assumptions of Theorem~\ref{t_dp5}}
We show that   $\lmin(\mc_{11})>\beta$, by bounding $\beta$ from above
and $\lmin(\mc_{11})$ from below.

Regarding the  upper bound for $\beta$, the expression (\ref{e_aux7}) and 
the assumption $\smin<\|\me\|_2$ imply
\begin{align}\label{e_aux2}
\beta=\left\|\begin{bmatrix} \ve_{12}^T & e_{22}+\smin& \ve_{32}^T\end{bmatrix}^T \right\|_2^2
\leq (\smin+\|\me\ve_n\|_2)^2\leq 4\|\me\|_2^2.
\end{align}
Regarding the lower bound for $\lmin(\mc_{11})$, view
$\mc_{11}=(\msig_1+\me_{11})^T(\msig_1+\me_{11})$
as a singular value problem,
so that $\lmin(\mc_{11})=\smin(\msig_1+\me_{11})^2$.
The well-conditioning of singular values \cite[Corollary 8.6.2]{GovL13} implies
\begin{equation*}
|\smin(\msig_1+\me_{11})-\smin(\msig_1)|\leq \|\me_{11}\|_2\leq \|\me\|_2.
\end{equation*}
Adding the assumption  
$\smin(\msig_1)=1/\|\msig_1^{-1}\|_2>4\|\me\|_2$ gives
\begin{equation*}
\smin(\msig_1+\me_{11})\geq \smin(\msig_1)-\|\me\|_2
>4\,\|\me\|_2-\|\me\|_2=3\|\me\|_2.
\end{equation*}
Now combine this lower bound for $\lmin(\mc_{11})$ with (\ref{e_aux2}),
\begin{equation*}
\lmin(\mc_{11})=\smin(\msig_1+\me_{11})^2>9\|\me\|_2^2>4\|\me\|_2^2\geq \beta.
\end{equation*}
Hence $\lmin(\mc_{11})>\beta$, and $\mc_{11}$ satisfies the assumptions of Theorem~\ref{t_dp5}.
\smallskip

\paragraph{2. Derive the  lower bound for $\tlmin$ from Theorem~\ref{t_dp5}}
In this bound,
\begin{align}\label{e_aux30}
\tlmin\geq \beta-\vb^T\mc_{11}^{-1}\vb-\frac{\beta\,\|\mc_{11}^{-1}\vb\|_2^2}{1-\beta\,\|\mc_{11}^{-1}\|_2},
\end{align}
where the key term is $\mc_{11}^{-1}\vb$. 
Insert the expression for $\vb$ from (\ref{e_aux7}),
\begin{align}
(\msig_1+\me_{11})^{-T}\vb&=
\ve_{12}+(\msig_1+\me_{11})^{-T}\left(\ve_{21}(\smin+e_{22})
+\me_{31}^T\ve_{32}\right)\notag\\
&=\ve_{12}+\vw.\label{eqn:pd01}
\end{align}
Combine the expression for $\mc_{11}$ from (\ref{e_C11}) with 
(\ref{eqn:pd01}),
\begin{align}\label{e_c11b}
\mc_{11}^{-1}\vb=(\msig_1+\me_{11})^{-1}
\underbrace{(\msig_1+\me_{11})^{-T}\vb}_{\ve_{12}+\vw}
=(\msig_1+\me_{11})^{-1}(\ve_{12}+\vw)
\end{align}
Multiply the above by $\vb^T$ on the left, and use (\ref{eqn:pd01})
\begin{align*}
\begin{split}
\vb^T\mc_{11}^{-1}\vb&=
\vb^T(\msig_1+\me_{11})^{-1}(\msig_1+\me_{11})^{-T}\vb=
(\ve_{12}+\vw)^T(\ve_{12}+\vw)\\
&=\|\ve_{12}+\vw\|_2^2
=\|\ve_{12}\|_2^2+2\ve_{12}^T\vw+\|\vw\|_2^2.
\end{split}
\end{align*}
Substitute the above, and  $\beta$ from (\ref{e_aux7}) into the first two summands 
of~(\ref{e_aux30}),
\begin{align}
\beta-\vb^T\mc_{11}^{-1}\vb&=(\|\ve_{12}\|_2^2+
(\smin+e_{22})^2+\|\ve_{32}\|_2^2)-
(\|\ve_{12}\|_2^2+2\ve_{12}^T\vw+\|\vw\|_2^2)\notag\\
&=\|\ve_{32}\|_2^2+(\smin+e_{22})^2-\underbrace{2\ve_{12}^T\vw}_{r_3}-
\|\vw\|_2^2.\label{e_aux30a}
\end{align}
Substitute the bound for $\beta$ in (\ref{e_aux2}),
and (\ref{e_c11b}) into the third summand of~(\ref{e_aux30}),
\begin{align}\label{e_aux30b}
\frac{\beta\,\|\mc_{11}^{-1}\vb\|_2^2}{1-\beta\,\|\mc_{11}^{-1}\|_2}\leq
4\frac{\|\me\|_2^2\,\|(\msig_1+\me_{11})^{-1}(\ve_{12}+\vw)\|_2^2}
{1-4\|\me\|_2^2\|(\msig_1+\me_{11})^{-1}\|_2^2}.
\end{align}
Inserting (\ref{e_aux30a}) and (\ref{e_aux30b}) into (\ref{e_aux30}) gives
\begin{align*}
\tlmin&\geq \|\ve_{32}\|_2^2+(\smin+e_{22})^2\\
&\quad -r_3-
\underbrace{\left(\|\vw\|_2^2+4\frac{\|\me\|_2^2\,\|(\msig_1+\me_{11})^{-1}(\ve_{12}+\vw)\|_2^2}
{1-4\|\me\|_2^2\|(\msig_1+\me_{11})^{-1}\|_2^2}\right)}_{r_4}.
\end{align*}
\end{proof}

\begin{remark}\label{r_det3}
The assumptions in Theorem~\ref{t_det3} are not restrictive.
Only a small gap of $3\|\me\|_2$ is required to 
separate the smallest singular value of $\ma$ from the remaining singular values,
\begin{align*}
\smin(\ma)<\|\me\|_2<4\|\me\|_2\leq 1/\|\msig_1^{-1}\|_2.  
\end{align*}
\end{remark}

\section{A cluster of small singular values}\label{s_cluster}
We extend the results in Section~\ref{s_single} from a single smallest singular value to a cluster of small singular values. 
To this end, we derive lower bounds for the small singular values of the perturbed matrix
 in terms of normwise absolute perturbations
 (Section~\ref{s_clb}),  based on eigenvalue bounds (Section~\ref{s_cev}).

\subsection{Auxiliary eigenvalue results}\label{s_cev}
We square the singular values of $\ma\in\rmn$ and consider instead the eigenvalues 
of the symmetric positive semi-definite matrix $\mb\equiv \ma^T\ma\in\rnn$.

For a symmetric positive semi-definite matrix $\mb\in\rnn$ 
with a cluster of $r$ small eigenvalues, we present an expression for these  eigenvalues
(Lemma~\ref{l_dp2}), and two lower bounds in terms
of normwise absolute perturbations (Theorems \ref{t_dp6} 
and~\ref{t_dp7}).

We assume that the $r$ small eigenvalues are  separated 
from the remaining ones, in the sense that they are strictly smaller than
the smallest eigenvalue of the leading principal submatrix $\mb_{11}$
of order $n-r$. The eigenvalues are labelled so that
\begin{align*}
\lambda_n(\mb) \leq \cdots \leq \lambda_{n-r+1}(\mb)<
\lambda_{n-r}(\mb)\leq \cdots \leq \lambda_1(\mb).
\end{align*}

The equality  below expresses the smallest eigenvalues in terms of themselves,
and represents an extension of Lemma~\ref{r_dp1} to clusters.

\begin{lemma}[Exact expression]\label{l_dp2}
Let $\mb\in\rnn$ be symmetric positive semi-definite with $\rank(\mb)\geq n-r$
for some $r\geq 1$, and partition
\begin{align*}
\mb=\begin{bmatrix}\mb_{11} & \mb_{12} \\ \mb_{12}^T & \mb_{22}\end{bmatrix} \qquad \text{where}\quad  \mb_{11}\in\real^{(n-r)\times (n-r)}, \quad \mb_{22}\in\real^{r\times r}.
\end{align*}
If  $\|\mb_{22}\|_2<\lmin(\mb_{11})$ then 
\begin{equation*}
\lambda_{n-r+j}(\mb)=\lambda_j\left(\mb_{22}-\mb_{12}^T(\mb_{11}-
\lambda_{n-r+j}(\mb)\,\mi)^{-1}\mb_{12}\right),
\qquad 1\leq j\leq r,
\end{equation*}
where 
\begin{equation}\label{e_dp2}
0\leq \lambda_{n-r+j}(\mb)\leq  \|\mb_{22}\|_2,\qquad 1\leq j\leq r.
\end{equation}
\end{lemma}

\begin{proof}
Abbreviate $\tl_{n-r+j}\equiv\lambda_{n-r+j}(\mb)$, $1\leq j\leq r$.
The lower bound in (\ref{e_dp2}) follows from the positive semi-definiteness of $\mb$, and the
upper bound  from the  Cauchy interlace theorem~\cite[Section~10.1]{Par80} 
\begin{align*}
\tl_{n-r+j}\leq \lambda_{j}(\mb_{22})\leq \lmax(\mb_{22})=\|\mb_{22}\|_2,\qquad 1\leq j\leq r.
\end{align*}
Combining this with the assumption $\|\mb_{22}\|_2<\lmin(\mb_{11})$
shows 
\begin{align}\label{e_aux8}
\tl_{n-r+j}\leq \|\mb_{22}\|_2<\lmin(\mb_{11}), \qquad 1\leq j\leq r.
\end{align}
Hence is $\mb_{11}-\tl_{n-r+j}\mi$ is nonsingular.

To derive the expression for $\tl_{n-r+j}$ , we start as in the proof of \cite[Theorem (10-1-2)]{Par80}.
The shifted matrix $\mb-\tl_{n-r+j}\mi$ has at most $n-r+j-1$ positive eigenvalues, at least one zero eigenvalue, and at most $r-j$ negative eigenvalues, $1\leq j\leq r$. Perform the congruence transformation
\begin{align*}
\mb-\tl_{n-r+j}\mi=\ml\begin{bmatrix}\mb_{11}-\tl_{n-r+j}\mi&\vzero \\ \vzero &\ms\end{bmatrix}\ml^T,\qquad
\ml\equiv\begin{bmatrix}\mi& \vzero \\ \mb_{12}^T(\mb_{11}-\tl_{n-r+j}\mi)^{-1} & \mi\end{bmatrix}
\end{align*}
where
\begin{align}\label{e_aux9}
\ms\equiv \mb_{22}-\tl_{n-r+j}\mi-\mb_{12}^T(\mb_{11}-\tl_{n-r+j}\mi)^{-1}\mb_{12},\qquad 
1\leq j\leq r.
\end{align}
From (\ref{e_aux8}) follows that $\mb_{11}-\tl_{n-r+j}\mi$ has $n-r$ positive eigenvalues. 
Combining this with the inertia preservation of congruence transformations implies that
$\ms$ has at most $r-j$ positive eigenvalues, at least one zero eigenvalue $\lambda_j(\ms)=0$, and at least $j-1$ negative eigenvalues, $1\leq j\leq r$.
Insert (\ref{e_aux9}) into $\lambda_j(\ms)=0$, and exploit the fact that the shift $\tl_{n-r+j}\,\mi$  
does not 
change the algebraic eigenvalue ordering, to obtain  the expression for $\tl_{n-r+j}$, 
$1\leq j\leq r$. 
\end{proof}


By restricting ourselves to a `dominant part' of $\mb_{11}$, we  weaken the expression in Lemma~\ref{l_dp2} to a lower bound, which allows the eigenvalues 
to be
negative.

\begin{lemma}[Lower bound]\label{l_dp2a}
Let $\mb\in\rnn$ be symmetric positive semi-definite with $\rank(\mb)\geq n-r$
for some $r\geq 1$, and partition
\begin{align*}
\mb=\begin{bmatrix}\mb_{11} & \mb_{12} \\ \mb_{12}^T & \mb_{22}\end{bmatrix}
\qquad \mathrm{where}\quad \mb_{11}\in\real^{(n-r)\times (n-r)},\quad
\mb_{22}\in\real^{r\times r}.
\end{align*}
Let $\mb_{11}=\mc_{11}+\mc_{12}$
where $\mc_{11}\in\real^{(n-r)\times (n-r)}$ is symmetric positive definite
and $\mc_{12}\in\real^{(n-r)\times (n-r)}$ is symmetric positive semi-definite. If
\begin{align*}
\widehat{\mb}\equiv\begin{bmatrix}\mc_{11} & \mb_{12} \\ \mb_{12}^T & \mb_{22}\end{bmatrix}
\end{align*}
with $\lmin(\mc_{11})>\|\mb_{22}\|_2$, then
\begin{align}
\lambda_{n-r+j}(\mb)&\geq \lambda_{n-r+j}(\widehat{\mb})\label{e_dp2b}\\
&= \lambda_j\left(\mb_{22}-\mb_{12}^T(\mc_{11}-
\lambda_{n-r+j}(\widehat{\mb})\,\mi)^{-1}\mb_{12}\right),
\quad 1\leq j\leq r,\label{e_dp2c}
\end{align}
where 
\begin{align}
\lambda_{n-r+j}(\widehat{\mb})&\leq  \|\mb_{22}\|_2,\label{e_dp2a}\\
\left\|\left(\mc_{11}-\lambda_{n-r+j}(\widehat{\mb})\,\mi\right)^{-1}\right\|_2
&\leq \frac{\|\mc_{11}^{-1}\|_2}{1-\|\mb_{22}\|_2\,\|\mc_{11}^{-1}\|_2},\qquad 1\leq j\leq r.
\label{e_dp2d}
\end{align}
\end{lemma}

\begin{proof}
The proof proceeds in four steps.

\paragraph{Proof of (\ref{e_dp2b})}
The symmetric positive semi-definiteness of $\mc_{12}$ and  Weyl's monotonicity
theorem \cite[Corollary 4.3.3]{HoJoI} imply
\begin{equation*}
\lambda_j(\mb)\geq \lambda_j(\widehat{\mb}), \qquad 1\leq j\leq n.
\end{equation*}
Now we concentrate on the eigenvalues of $\widehat{\mb}$, and
abbreviate $\twh_{n-r+j}\equiv\lambda_{n-r+j}(\widehat{\mb)}$, $1\leq j\leq r$.

\paragraph{Proof of (\ref{e_dp2a})} Apply
the  Cauchy interlace theorem~\cite[Section~10.1]{Par80} to~$\widehat{\mb}$,
\begin{align*}
\twh_{n-r+j}\leq \lambda_{j}(\mb_{22})\leq \lmax(\mb_{22})=\|\mb_{22}\|_2,\qquad 1\leq j\leq r.
\end{align*}
Combining this with the assumption $\|\mb_{22}\|_2<\lmin(\mc_{11})$
shows 
\begin{align*}
\twh_{n-r+j}\leq \|\mb_{22}\|_2<\lmin(\mc_{11}), \qquad 1\leq j\leq r.
\end{align*}
Hence  $\mc_{11}-\twh_{n-r+j}\mi$ is nonsingular, which holds in particular if
$\twh_{n-r+j}<0$.

\paragraph{Proof of  (\ref{e_dp2c})} 
To derive the expression for $\twh_{n-r+j}$ , apply the proof of
Lemma~\ref{l_dp2} to the eigenvalues of $\widehat{\mb}$. 
This proof relies only on the signs of eigenvalues of
shifted matrices,
and does not require positive semi-definiteness of the host matrix $\widehat{\mb}$.

\paragraph{Proof of (\ref{e_dp2d})} 
Fix some $1\leq j\leq r$ for the inverse in (\ref{e_dp2c}). Then
factor out $\mc_{11}^{-1}$,
\begin{align*}
(\mc_{11}-\twh_{n-r+j}\,\mi)^{-1}=
\mc_{11}^{-1}\md \qquad \textrm{where}\quad \md\equiv (\mi-\twh_{n-r+j}\,\mc_{11}^{-1})^{-1},
\end{align*}
and take norms,
\begin{align*}
\left\|(\mc_{11}-\twh_{n-r+j}\,\mi)^{-1}\right\|_2\leq
\|\mc_{11}^{-1}\|_2\,\|\md\|_2.
\end{align*}
To bound $\|\md\|_2$, consider the eigenvalue decomposition
 $\mc_{11}=\mw\mlam\mw^T$, where $\mw$ is an orthogonal matrix, 
 and the diagonal matrix
 \begin{align*}
 \mlam=\diag\begin{pmatrix} \gamma_1 & \cdots & \gamma_{n-r}\end{pmatrix}
 \in\real^{(n-r)\times (n-r)}
 \end{align*}
  has positive diagonal elements $\gamma_{\ell}>0$.
Thus $\md$ has an eigenvalue decomposition
$\md=\mw (\mi-\twh_{n-r+j}\,\mlam^{-1})^{-1}\mw^T$
with eigenvalues  
\begin{align*}
\lambda_{\ell}(\md)=1/\left(1-\frac{\twh_{n-r+j}}{\gamma_\ell}\right),\qquad 1\leq \ell\leq n-r.
\end{align*}
Case 1: If $\twh_{n-r+j}\geq 0$, then (\ref{e_dp2a}) implies 
\begin{align*}
0\leq \frac{\twh_{n-r+j}}{\gamma_{\ell}}\leq \frac{\|\mb_{22}\|_2}{\lmin(\mc_{11})}=
\|\mb_{22}\|_2\|\mc_{11}^{-1}\|_2<1, \qquad 1\leq \ell\leq n-r.
\end{align*}
Hence
\begin{align}\label{e_aux8b}
\|\md\|_{22}=\max_{1\leq \ell\leq n-r}{|\lambda_j(\md)|}\leq \frac{1}{1-\|\mb_{22}\|_2\|\mc_{11}^{-1}\|_2}.
\end{align}
Case 2: If $\twh_{n-r+j}<0$, then $\gamma_{\ell}>0$  and (\ref{e_dp2a}) imply
\begin{align*}
1-\frac{\twh_{n-r+j}}{\gamma_\ell}=1+\frac{|\twh_{n-r+j}|}{\gamma_\ell}>1
>1-\|\mb_{22}\|_2\|\mc_{11}^{-1}\|_2,\qquad 1\leq \ell\leq n-r.
\end{align*}
Again, as in (\ref{e_aux8b}) we conclude
\begin{align*}
\|\md\|_{22}=\max_{1\leq \ell\leq n-r}{|\lambda_j(\md)|}\leq \frac{1}{1-\|\mb_{22}\|_2\|\mc_{11}^{-1}\|_2}.
\end{align*}
Since we fixed an arbitrary $j$ to show (\ref{e_dp2d}), it holds for all $1\leq j\leq r$.
\end{proof}

The subsequent lower bounds are informative if the offdiagonal
part has small norm. We start with an extension of Theorem~\ref{t_dp4}.

\begin{theorem}[First lower bound]\label{t_dp6}
Let $\mb\in\rnn$ be symmetric positive semi-definite  with $\rank(\mb)\geq n-r$
for some $r\geq 1$, and partition
\begin{align*}
\mb=\begin{bmatrix}\mb_{11} & \mb_{12} \\ \mb_{12}^T & \mb_{22}\end{bmatrix} \qquad \text{where}\quad  \mb_{11}\in\real^{(n-r)\times (n-r)}, \quad \mb_{22}\in\real^{r\times r}.
\end{align*}
If  $\|\mb_{22}\|_2<\lmin(\mb_{11})$ then 
$\lambda_{n-r+j}(\mb)\geq\lambda_j(\mz_j)$, $1\leq j\leq r$, where
\begin{equation*}
\mz_j\equiv \mb_{22}-\mb_{12}^T\mb_{11}^{-1}\mb_{12}-\|\mb_{22}\|_2\,\mb_{12}^T\mb_{11}^{-1}(\mi-\lambda_{n-r+j}(\mb)\,\mb_{11}^{-1})^{-1}\mb_{11}^{-1}\mb_{12}.
\end{equation*}
\end{theorem}

\begin{proof}
Abbreviate $\tl_{n-r+j}\equiv\lambda_{n-r+j}(\mb)$, $1\leq j\leq r$.
As in the proof of Theorem~\ref{t_dp4}, apply the Sherman-Morrison formula 
\cite[Section 2.1.4]{GovL13}
\begin{equation*}
(\mb_{11}-\tl_{n-r+j}\mi)^{-1}=\mb_{11}^{-1}+\tl_{n-r+j}\,\mb_{11}^{-1}(\mi-\tl_{n-r+j}\,
\mb_{11}^{-1})^{-1}\mb_{11}^{-1},
\end{equation*}
and substitute the above
into the expressions for
\begin{align}\label{e_dp6a}
\tl_{n-r+j}=\lambda_j(\mm_j), \qquad 1\leq j\leq r
\end{align}
from Lemma~\ref{l_dp2} where
\begin{align*}
\mm_j&\equiv \mb_{22}-\mb_{12}^T(\mb_{11}-\tl_{n-r+j}\mi)^{-1}\mb_{12}\\
&=\mb_{22}-\mb_{12}^T\mb_{11}^{-1}\mb_{12}-\tl_{n-r+j}\mb_{12}^T\mb_{11}^{-1}(\mi-\tl_{n-r+j}\mb_{11}^{-1})^{-1}\mb_{11}^{-1}\mb_{12}.
\end{align*}
From (\ref{e_dp2}) follows the Loewner bound
\begin{align*}
\mm_j\succeq
\mz_j\equiv\mb_{22}-\mb_{12}^T\mb_{11}^{-1}\mb_{12}-\|\mb_{22}\|_2\,\mb_{12}^T\mb_{11}^{-1}(\mi-\tl_{r+j}\mb_{11}^{-1})^{-1}\mb_{11}^{-1}\mb_{12}.
\end{align*}
This and (\ref{e_dp6a}) imply
 $\tl_{n-r+j}=\lambda_j(\mm_j)\geq \lambda_j(\mz_j)$, $1\leq j\leq r$ \cite[Corollary 7.7.4]{HoJoI}.
\end{proof}

The slightly weaker bound below extends Theorem~\ref{t_dp5} and focusses on a 'dominant'
part of $\mb_{11}$. This establishes the connection to
Theorem~\ref{t_det4}, where $\mb$
represents the perturbed matrix and 
the low order terms in $\mb_{11}$ are captured by $\mc_{11}$.

\begin{theorem}[Second lower bound]\label{t_dp7}
Let $\mb\in\rnn$ be symmetric positive semi-definite  with $\rank(\mb)\geq n-r$
for some $r\geq 1$, and partition
\begin{align*}
\mb=\begin{bmatrix}\mb_{11} & \mb_{12} \\ \mb_{12}^T & \mb_{22}\end{bmatrix} \qquad \text{where}\quad  \mb_{11}\in\real^{(n-r)\times (n-r)}, \quad \mb_{22}\in\real^{r\times r}.
\end{align*}
If $\mb_{11}=\mc_{11}+\mc_{12}$ where  $\mc_{11}$ is symmetric positive definite with
$\lmin(\mc_{11})>\|\mb_{22}\|_2$, and $\mc_{12}$
is symmetric positive semi-definite, then 
\begin{equation*}
\lambda_{n-r+j}(\mb)\geq\lambda_j\left(\mb_{22}-\mb_{12}^T\mc_{11}^{-1}\mb_{12}\right)
-\frac{\|\mb_{22}\|_2\,\|\mc_{11}^{-1}\mb_{12}\|_2^2}{1-\|\mb_{22}\|_2\,\|\mc_{11}^{-1}\|_2},
\qquad 1\leq j\leq r.
\end{equation*}
\end{theorem}

\begin{proof}
Define
\begin{align*}
\widehat{\mb}\equiv \begin{bmatrix} \mc_{11} & \mb_{12}\\
\mb_{12}^T & \mb_{22}\end{bmatrix},
\end{align*}
and abbreviate $\twh_{n-r+j}\equiv\lambda_{n-r+j}(\widehat{\mb})$, $1\leq j\leq r$.
From (\ref{e_dp2c}) in Lemma~\ref{l_dp2a} follows
\begin{align*}
\lambda_{n-r+j}(\mb)\geq \twh_{n-r+j}=
\lambda_j\left(\mb_{22}-\mb_{12}^T(\mc_{11}-
\twh_{n-r+j}\,\mi)^{-1}\mb_{12}\right),
\quad 1\leq j\leq r,\end{align*}
We proceed as in the proof of Theorem~\ref{t_dp6}, and
apply the Sherman-Morrison formula  \cite[Section 2.1.4]{GovL13},
\begin{equation*}
(\mc_{11}-\twh_{n-r+j}\mi)^{-1}=
\mc_{11}^{-1}+\twh_{n-r+j}\,\mc_{11}^{-1}(\mi-\twh_{n-r+j}\,
\mc_{11}^{-1})^{-1}\mc_{11}^{-1},
\end{equation*}
 and  (\ref{e_dp2a}) to the expression for
\begin{align}\label{e_dp7a}
\twh_{n-r+j}=\lambda_j(\mm_j), \qquad 1\leq j\leq r,
\end{align}
from Lemma~\ref{l_dp2a}, where
\begin{align*}
\mm_j&\equiv
 \mb_{22}-\mb_{12}^T(\mc_{11}-\twh_{n-r+j}\mi)^{-1}\mb_{12}, \qquad 1\leq j\leq r\\
&=\mb_{22}-\mb_{12}^T\mc_{11}^{-1}\mb_{12}-\tl_{n-r+j}\mb_{12}^T\mc_{11}^{-1}(\mi-\tl_{n-r+j}
\mc_{11}^{-1})^{-1}\mc_{11}^{-1}\mb_{12}.
\end{align*}
From (\ref{e_dp2a}) and (\ref{e_dp2d}) follows the lower bound
\begin{align*}
\mm_j&\succeq \mb_{22}-\mb_{12}^T\mc_{11}^{-1}\mb_{12}-\|\mb_{22}\|_2\,\mb_{12}^T\mc_{11}^{-1}
(\mi-\tl_{r+j}\mc_{11}^{-1})^{-1}\mc_{11}^{-1}\mb_{12}\\
&\succeq \mz\equiv \mb_{22}-\mb_{12}^T\mc_{11}^{-1}\mb_{12}-
\underbrace{\frac{\|\mb_{22}\|_2\,\|\mc_{11}^{-1}\mb_{12}\|_2^2}{1-\|\mb_{22}\|_2\,\|\mc_{11}^{-1}\|_2}}_{\gamma} \mi, \qquad 1\leq j\leq r.
\end{align*}
Thus, $\mm_j\succeq \mz$, $1\leq j\leq r$.
 The Loewner properties \cite[Corollary 7.7.4]{HoJoI} imply 
 the same for the eigenvalues, $\lambda_j(\mm_j)\geq \lambda_j(\mz)$, $1\leq j\leq r$.
 Combine this with the well conditioning
 of eigenvalues \cite[Theorem 8.1.5]{GovL13},
\begin{align*}
\tl_{n-r+j}=\lambda_j(\mm_j)\geq \lambda_j(\mz)\geq \lambda_j(\mb_{22}-\mb_{12}^T\mb_{11}^{-1}\mb_{12})-\gamma,\qquad 1\leq j\leq r.
\end{align*}
\end{proof}

Theorem~\ref{t_dp7} reduces to Theorem~\ref{t_dp5} for $r=1$, and to Theorem~\ref{t_dp6} for 
$\mc_{12}=\vzero$.

\subsection{A lower bound for a cluster of smallest singular values}\label{s_clb}
We extend the bound for a single smallest singular value in section~\ref{s_slb} to a cluster of smallest singular values.
The resulting  lower bound for the cluster of
perturbed smallest singular values (Theorem~\ref{t_det4}) is based on the eigenvalue bounds in section~\ref{s_cev}, 
and expressed  in terms of normwise absolute perturbations.
We start with a summary of all assumptions (Assumptions~\ref{ass_2}), and end with a 
discussion of their generality (Remark~\ref{r_det4}).

\begin{assumptions}\label{ass_2}
Let $\ma\in\rmn$ with $m\geq n$ have $\rank(\ma)\geq n-r$ for some $r\geq 1$.
Let $\ma=\mU\msig\mv^T$ be a full singular value decomposition,
where $\msig\in\rmn$ is diagonal, and
$\mU\in\rmm$ and $\mv\in\rnn$ are orthogonal matrices. 
Partition commensurately,
\begin{align*}
\msig=\begin{bmatrix} \msig_1 & \vzero \\ \vzero & \msig_2\\ \vzero & \vzero \end{bmatrix},\qquad 
\me=\mU\begin{bmatrix} \me_{11} & \me_{12} \\ \me_{21} & \me_{22}\\ \me_{31}& \me_{32} \end{bmatrix}\mv^T, 
\end{align*}
where $\msig_1\in\real^{(n-r)\times (n-r)}$ is nonsingular diagonal, 
and $\msig_2\in\real^{r\times r}$ is diagonal.
\end{assumptions}

This bound below extends Theorem~\ref{t_det3}, and reduces to it for $r=1$.

\begin{theorem}\label{t_det4}
Let $\ma,\me\in\rmn$ satisfy Assumptions~\ref{ass_2}.
If $1/\|\msig_{1}^{-1}\|_2>4\|\me\|_2$ and $\|\msig_2\|_2<\|\me\|_2$,
then 
\begin{align*}
\sigma_{n-r+j}(\ma+\me)^2\geq\lambda_j(\me_{32}^T\me_{32}+(\msig_2+\me_{22})^T(\msig_2+\me_{22}) -\mr_3)-r_4,\quad 1\leq j\leq r,
\end{align*}
where $\mr_3$ contains terms of order 3
\begin{align*}
\mr_3\equiv \me_{12}^T\mw+\mw^T\me_{12},\qquad 
\mw\equiv(\msig_1+\me_{11})^{-T}\begin{bmatrix}\me_{21}& \me_{31}^T\end{bmatrix}
\begin{bmatrix} \msig_2+\me_{22}\\ \me_{32}\end{bmatrix}
\end{align*}
and $r_4$ contains terms of order 4 and higher,
\begin{align*}
r_4\equiv \|\mw\|_2^2
+4\frac{\|\me\|_2^2\,\|(\msig_1+\me_{11})^{-1}(\me_{12}+\mw)\|_2^2}
{1-4\|\me\|_2^2 \|(\msig_1+\me_{11})^{-1}\|_2^2}.
\end{align*}
\end{theorem}

\begin{proof}
We square the singular values of $\ma+\me$ and consider the eigenvalues of the perturbed
matrix
\begin{align*}
\mb\equiv (\ma+\me)^T(\ma+\me)=\mv\begin{bmatrix} \mb_{11} & \mb_{12} \\ \mb_{12}^T & \mb_{22}\end{bmatrix}\mv^T
\end{align*}
where
\begin{align}
\mb_{11}&=\underbrace{(\msig_1+\me_{11})^T(\msig_1+\me_{11})}_{\mc_{11}}
+\underbrace{\me_{21}^T\me_{21}+\me_{31}^T\me_{31}}_{\mc_{12}}\label{e_det4aa}\\
\begin{split}\label{e_det4a}
\mb_{22}&=\me_{12}^T\me_{12}+(\msig_{2}+\me_{22})^T(\msig_{2}+\me_{22})
+\me_{32}^T\me_{32}\\
\mb_{12}& = (\msig_1+\me_{11})^T\me_{12}+\me_{21}^T(\msig_{2}+\me_{22}) +\me_{31}^T\me_{32}.
\end{split}
\end{align}
From $\smin(\msig_1)> 4\|\me\|_2$ follows that
$\mc_{11}$ is  symmetric positive definite, while $\mc_{12}$ is symmetric positive semi-definite and contains only second order terms.
Abbreviate $\tl_{n-r+j}\equiv \lambda_{n-r+j}(\mb)=\sigma_{n-r+j}(\ma+\me)^2$, $1\leq j\leq r$.

The proof proceeds in two steps:
\begin{enumerate}
\item Confirming that  $\mc_{11}$ satisfies the assumptions of Theorem~\ref{t_dp7}.
\item Deriving the lower bounds for $\tl_{n-r+j}$ from Theorem~\ref{t_dp7}.
\end{enumerate}
\smallskip

\paragraph{1. Confirm that  $\mc_{11}$ satisfies the assumptions of Theorem~\ref{t_dp7}}
We show that   $\lmin(\mc_{11})>\|\mb_{22}\|_2$, by bounding $\|\mb_{22}\|_2$ from above
and $\lmin(\mc_{11})$ from below.

Regarding the  upper bound for $\|\mb_{22}\|_2$, the expression 
for $\mb_{22}$ in (\ref{e_det4a}) and 
the assumption $\|\msig_2\|_2<\|\me\|_2$ imply
\begin{align}\label{e_aux2c}
\|\mb_{22}\|_2=\left\|\begin{bmatrix} \me_{12}^T & (\me_{22}+\msig_2)^T& \me_{32}^T\end{bmatrix}^T \right\|_2^2
\leq (\|\msig_2\|_2+\|\me\|_2)^2\leq 4\|\me\|_2^2.
\end{align}
Regarding the lower bound for $\lmin(\mc_{11})$, view
$\mc_{11}=(\msig_1+\me_{11})^T(\msig_1+\me_{11})$
as a singular value problem,
so that $\lmin(\mc_{11})=\smin(\msig_1+\me_{11})^2$.
The well-conditioning of singular values \cite[Corollary 8.6.2]{GovL13} implies
\begin{equation*}
|\smin(\msig_1+\me_{11})-\smin(\msig_1)|\leq \|\me_{11}\|_2\leq \|\me\|_2.
\end{equation*}
Adding the assumption  
$\smin(\msig_1)=1/\|\msig_1^{-1}\|_2>4\|\me\|_2$ gives
\begin{equation*}
\smin(\msig_1+\me_{11})\geq \smin(\msig_1)-\|\me\|_2
>4\,\|\me\|_2-\|\me\|_2=3\|\me\|_2.
\end{equation*}
Now combine this lower bound for $\lmin(\mc_{11})$ with (\ref{e_aux2c}),
\begin{equation*}
\lmin(\mc_{11})=\smin(\msig_1+\me_{11})^2>9\|\me\|_2^2>4\|\me\|_2^2\geq \|\mb_{22}\|_2.
\end{equation*}
Hence $\lmin(\mc_{11})>\|\mb_{22}\|_2$, and $\mc_{11}$ satisfies the assumptions of Theorem~\ref{t_dp5}.
\smallskip

\paragraph{2. Derive the  lower bounds for $\tl_{n-r+j}$ from Theorem~\ref{t_dp7}}
In these bounds,
\begin{align}\label{e_aux40}
\tl_{n-r+j}\geq\lambda_j(\ms)
-\frac{\|\mb_{22}\|_2\,\|\mc_{11}^{-1}\mb_{12}\|_2^2}{1-\|\mb_{22}\|_2\,\|\mc_{11}^{-1}\|_2},\qquad
\ms\equiv \mb_{22}-\mb_{12}^T\mc_{11}^{-1}\mb_{12},
\end{align}
the key term is $\mc_{11}^{-1}\mb_{12}$. 
Insert the expression for $\mb_{12}$ from (\ref{e_det4a}),
\begin{align}
(\msig_1+\me_{11})^{-T}\mb_{12}&=
\me_{12}+(\msig_1+\me_{11})^{-T}\left(\me_{21}^T(\msig_2+\me_{22})
+\me_{31}^T\me_{32}\right)\nonumber\\
&=\me_{12}+\mw.\label{eqn:pd02}
\end{align}
Combine the expression for $\mc_{11}$ from (\ref{e_det4aa})
with the above,
\begin{align}\label{e_c11b2}
\mc_{11}^{-1}\mb_{12}&=(\msig_1+\me_{11})^{-1}
\underbrace{(\msig_1+\me_{11})^{-T}\mb_{12}}_{\me_{12}+\mw}=
(\msig_1+\me_{11})^{-1}(\me_{12}+\mw)
\end{align}
Multiply the above by $\mb_{12}^T$ on the left,
and use (\ref{eqn:pd02}),
\begin{align*}
\mb_{12}^T\mc_{11}^{-1}\mb_{12}&=
\mb_{12}^T(\msig_1+\me_{11})^{-1}(\msig_1+\me_{11})^{-T}\mb_{12}\\
&=(\me_{12}+\mw)^T(\me_{12}+\mw)=\me_{12}^T\me_{12}+\me_{12}^T\mw+\mw^T\me_{12}+\mw^T\mw.
\end{align*}
Substitute the above, and $\mb_{22}$ from (\ref{e_det4a}) into $\ms$
from (\ref{e_aux40}),
\begin{align*}
\ms&=
\me_{12}^T\me_{12}+(\msig_{2}+\me_{22})^T(\msig_{2}+\me_{22})+\me_{32}^T\me_{32}\\
&\qquad\qquad\qquad -
(\me_{12}^T\me_{12}+\me_{12}^T\mw+\mw^T\me_{12}+\mw^T\mw)\\
&=\me_{32}^T\me_{32}+(\msig_{2}+\me_{22})^T(\msig_{2}+\me_{22})
-\mr_3-\mw^T\mw.
\end{align*}
The well conditioning of eigenvalues \cite[Theorem 8.1.5]{GovL13} implies
\begin{align}\label{e_aux40a}
\lambda_j(\ms)
\geq \lambda_j(\me_{32}^T\me_{32}+(\msig_{2}+\me_{22})^T(\msig_{2}+\me_{22})
-\mr_3)-\|\mw\|_2^2.
\end{align}
Substitute the bound for $\|\mb_{22}\|_2$ from (\ref{e_aux2c}),
and (\ref{eqn:pd02}) into the second summand of~(\ref{e_aux40}),
\begin{align}\label{e_aux40b}
\frac{\|\mb_{22}\|_2\,\|\mc_{11}^{-1}\mb_{12}\|_2^2}{1-\|\mb_{22}\|_2\,\|\mc_{11}^{-1}\|_2}
\leq 4\frac{\|\me\|_2^2\,\|(\msig_1+\me_{11})^{-1}(\me_{12}+\mw)\|_2^2}
{1-4\|\me\|_2^2\|(\msig_1+\me_{11})^{-1}\|_2^2}.
\end{align}
At last insert (\ref{e_aux40a}) and (\ref{e_aux40b}) into (\ref{e_aux40}).
\end{proof}

\begin{remark}\label{r_det4}
The assumptions in Theorem~\ref{t_det4} are not restrictive.
Only a small gap of $3\|\me\|_2$ is required to 
separate the small singular value cluster of $\ma$ from the remaining singular values,
\begin{align*}
\|\msig_2\|_2<\|\me\|_2<4\|\me\|_2\leq 1/\|\msig_1^{-1}\|_2.  
\end{align*}
\end{remark} 
\section{Numerical experiments}\label{s_exp}
We present numerical experiments to illustrate that downcasting
to lower precision can increase small singular values, thus confirming 
that our bounds in sections~\ref{s_single} and~\ref{s_cluster}
are informative models for the effects of reduced arithmetic
precision.

After describing the algorithms for computing
the singular values (Section~\ref{s_setup}), we present
the numerical experiments (Section~\ref{s_results}).

\subsection{Generation and computation of singular values}\label{s_setup}
The code for the numerical experiments consists of two algorithms:  
Algorithm~\ref{al2} in Appendix~\ref{appen}
generates the diagonal matrix $\msig$ containing the exact singular values, while Algorithm~\ref{al1} generates the matrix~$\ma\in\rmn$ from $\msig$ in double precision, and 
then computes the singular values of $\ma$ and those of its 
lower precision versions $\mathtt{single}(\ma)$ and 
$\mathtt{half}(\ma)$. We use \texttt{Julia} programming language for our computations. The scripts for reproducing the numerical experiment are published in our git 
repository\footnote{\url{https://github.com/cboutsikas/small_sigmas_increase.git}}.

The $n$ singular values in the diagonal matrix $\msig$ generated by Algorithm~\ref{al2}
consist of two clusters: a cluster $\msig_1$ of large singular values, and a cluster $\msig_2$ of small singular values.
Each cluster is defined by the following input parameters:
the number of singular values; the smallest and largest singular value; and the gap between the two clusters. 
Specifically, cluster $\msig_{1}$ consists of $k_1>0$ singular values, the largest 
one being $10^{s_{1}}$ and the smallest one being $10^{s_1 - d_1}$. Here
$d_1\geq 0$ controls the distance
between smallest and largest singular value. If $d_1>0$ and $k_1>2$, then the interior singular values 
of $\msig_1$ are
sampled uniformly at random in the
interval $[10^{s_1-d_1}, 10^{s_1}]$.

The parameter $g$ controls the gap between the clusters, which is set to $10^g$.
Cluster $\msig_2$ consists of $k_2\equiv n-k_1\geq 0$ singular values,
the largest one being $10^{s_{1}-d_1-g}$, 
and the smallest one being $10^{s_1-d_1-g-d_2}$,
where $d_2\geq 0$ controls the distance between the smallest and largest singular value
in cluster $\msig_2$.
If $d_2>0$ and $k_2>2$, then the interior singular values of $\msig_2$ are
sampled uniformly at random in the
interval $[10^{s_1-d_1-g-d_2}, 10^{s_1-d_1-g}]$.

\subsection{Numerical results and discussion}\label{s_results}
We present numerical experiments that
corroborate bounds in sections \ref{s_single} and~\ref{s_cluster}. 
Our experiments are performed on matrices $\ma \in \real^{m \times n}$ with $\rank(\ma)=n$, 
$m = 4,096$ and $n=256$. We emphasize that changing the matrix dimensions while keeping the aspect ratio $m/n$ fixed does not change our conclusions. 

Figures \ref{fig4.1}--\ref{fig4.4} show the exact singular values as well as the singular values computed in double precision, which turn out to be identical in all cases.
In addition, Figures~\ref{fig4.1} and~\ref{fig4.3} show the singular values computed in \textit{single precision}, while Figures~\ref{fig4.2} and~\ref{fig4.4} show the singular values computed in \textit{half precision}. 

To guarantee that our empirical evaluations satisfy Assumptions~\ref{ass_1} and~\ref{ass_2},
we compute the singular values with the Golub-Kahan Bi-Diagonalization\footnote{\texttt{Julia} computes the SVD with the LAPACK routine \texttt{dgesvd()}, which employs the Bi-Diagonalization method.} algorithm~\cite{golub1965calculating} in the respective precision. As in~\cite[Algorithm 8.6.2]{GovL13} we assume that
$$\mU^{T} \ma \mv = \msig + \me, $$ 
satisfies $\| \me \|_{2} \approx u\| \ma \|_{2}$ , where $u$ is the unit roundoff 
\cite[section]{GovL13}, which
depends on the underlying precision\footnote{In double precision (binary64), $u= 2^{-53} \approx 1.11\times 10^{-16}$; in single precision (binary32), $u = 2^{-24} \approx 5.96 \times 10^{-8}$; and in half precision (binary16) $u = 2^{-11} \approx 4.88 \times 10^{-4}$.}. Thus, we can express the assumption in Remark~\ref{r_det4} as follows:
\begin{align}\label{eqn:aspd}
\sigma_{n-r+1} = \|\msig_2\|_2 \ \lesssim \ \sigma_{\max}\,u \ \lesssim \ 4 \,\sigma_{\max} u
\ \lesssim 1/\|\msig_1^{-1}\|_2 = \sigma_{n-r},
\end{align}
for some $r\geq 1$. Table~\ref{table4.1} shows the increase in the smallest singular value for $r=1$; and the average increase in the $r$ smallest singular values $r>1$. Since we describe a qualitative model, an increase can be observed even when the assumptions are not satisfied. We provide demonstrations of two such cases in Appendix~\ref{appen_B}.
%

\begin{table}[H]
\centering
\begin{tabular}{|l|l|l|l|l|l|l|}
\hline
 &$\min(\msig)$ & $\min(\msig^{d})$  & $\min(\msig^{s})$ & $\texttt{avg}(\msig^{s}_2)$  & $\min(\msig^{h})$ & $\texttt{avg}(\msig^{h}_2)$  \\
  \hline
Fig. \ref{fig4.1} & $10^{-7}$  & $10^{-7}$  & $6 \times 10^{-7}$ & $5 \times 10^{-7}$  & N/A & N/A  \\
 \hline
 Fig. \ref{fig4.2} & $10^{-3}$  & $10^{-3}$  & N/A & N/A  & $4 \times 10^{-3}$  & $3 \times 10^{-3}$  \\
  \hline
 Fig. \ref{fig4.3} & $10^{-5}$  & $10^{-5}$  & $6 \times 10^{-4}$ & $4 \times 10^{-4}$  & N/A & N/A \\
  \hline
 Fig. \ref{fig4.4} & $10^{-4}$  & $10^{-4}$  & N/A & N/A  & $8 \times 10^{-3}$ & $6 \times 10^{-3}$ \\
  \hline
\end{tabular}
\caption{Smallest singular values in Figures \ref{fig4.1}--\ref{fig4.4}: 
exact ($\msig$); double precision ($\msig^d$); single precision ($\msig^s$); 
and half precision ($\msig^h$). The quantities $\texttt{avg}(\msig^{s}_2)$ and $\texttt{avg}(\msig^{h}_2)$ represent the average increase of the $r$ smallest singular values for single and half precision.}
\label{table4.1}
\end{table}

\subsubsection{A single smallest singular value}\label{s_esingle}
We illustrate that downcasting to lower precision
can increase the smallest singular value, thus confirming that
Theorem~\ref{t_det3} represents a proper qualitative model
for the effects of reduced precision.
In Figures~\ref{fig4.1} and \ref{fig4.2}, the small singular value cluster $\msig_2$ 
consists of a single singular value, while the large singular value cluster $\msig_1$ contains 255 singular values.

\paragraph{Figure~\ref{fig4.1}}
The cluster $\msig_1$ contains 255 distinct singular values in the interval $[10^{-4}, 10^2]$, while $\msig_2$ contains the single singular value $10^{-7}$. The values in Assumption~\ref{ass_1} and (\ref{eqn:aspd}) are
\begin{align*}
    \underbrace{\sigma_{n}}_{10^{-7}}\  \lesssim\  \underbrace{\sigma_{\max}\, u}_{6 \times 10^{-6}} \ \lesssim \ \underbrace{4\,\sigma_{\max}\,u}_{2.4 \times 10^{-5}} \ \lesssim\  
    \underbrace{\sigma_{n-1}}_{10^{-4}}.
\end{align*}
In single precision, the smallest singular value has increased by almost $5\times 10^{-7}$.

\paragraph{Figure~\ref{fig4.2}}
The cluster $\msig_1$ contains 255 distinct singular values
in the interval $[10^{-1}, 10^1]$, while $\msig_2$ contains the single singular value $10^{-3}$.
The values in Assumption~\ref{ass_1} and~(\ref{eqn:aspd})) are
\begin{align*}
    \underbrace{\sigma_{n}}_{10^{-3}}\  \lesssim\  \underbrace{\sigma_{\max}\, u}_{5 \times 10^{-3}} \ \lesssim\  \underbrace{4\, \sigma_{\max}\, u}_{2 \times 10^{-2}} \ \lesssim\ \underbrace{\sigma_{n-1}}_{10^{-1}}.
\end{align*}
In half precision, the smallest singular value has increased
to $4\times 10^{-3}$.

\subsubsection{A cluster of small singular values}\label{s_ecluster}
We illustrate that downcasting to lower precision
can increase the values of the  cluster of small singular values, thus
confirming that
 Theorem~\ref{t_det4} represents a proper qualitative model for
 the effects of reduced precision.
In Figures~\ref{fig4.3} and \ref{fig4.4}, 
the small singular value cluster $\msig_2$ contains 28 singular values, 
while the large singular value cluster $\msig_1$ contains 228 singular values.

\paragraph{Figure~\ref{fig4.3}}
The cluster $\msig_1$ contains 228 distinct singular values
in the interval $[10^{-1}, 10^5]$, while $\msig_2$ contains 28 singular values
in the interval $[10^{-5}, 10^{-3}]$. 
The values in Assumption~\ref{ass_1} and ~(\ref{eqn:aspd}) are
\begin{align*}
    \underbrace{\sigma_{n-r}}_{10^{-3}} \ \lesssim\ \underbrace{\sigma_{\max}\, u}_{6 \times 10^{-3}} \ \lesssim\  \underbrace{4\, \sigma_{\max}\, u}_{2.4 \times 10^{-2}}\  \lesssim\  \underbrace{\sigma_{n-r+1}}_{10^{-1}}.
\end{align*}
In single precision, the smallest singular value of $\msig_2$ has increased to $6 \times 10^{-4}$, with an average increase of $4 \times 10^{-4}$ for the $r$ smallest singular values.

\paragraph{Figure~\ref{fig4.4}}
The cluster $\msig_1$ contains 228 distinct singular values
in the interval $[10^{0}, 10^2]$, while $\msig_2$ contains 28 singular values
in the interval $[10^{-4}, 10^{-2}]$. 
The values in Assumption~\ref{ass_1} and~(\ref{eqn:aspd}) are
\begin{align*}
    \underbrace{\sigma_{n-r}}_{10^{-2}}\ \lesssim\  \underbrace{\sigma_{\max}\, u}_{5 \times 10^{-2}} \ \lesssim\  \underbrace{4\, \sigma_{\max}\, u}_{10^{-1}} \ \lesssim\  \underbrace{\sigma_{n-r+1}}_{10^{0}}.
\end{align*}
In half precision, the smallest singular value of $\msig_2$ has increased to $8 \times 10^{-3}$, with an average increase of $7\times 10^{-3}$ for the $r$ smallest singular values.

\begin{figure}[ht!]
\includegraphics[width=0.85\textwidth]{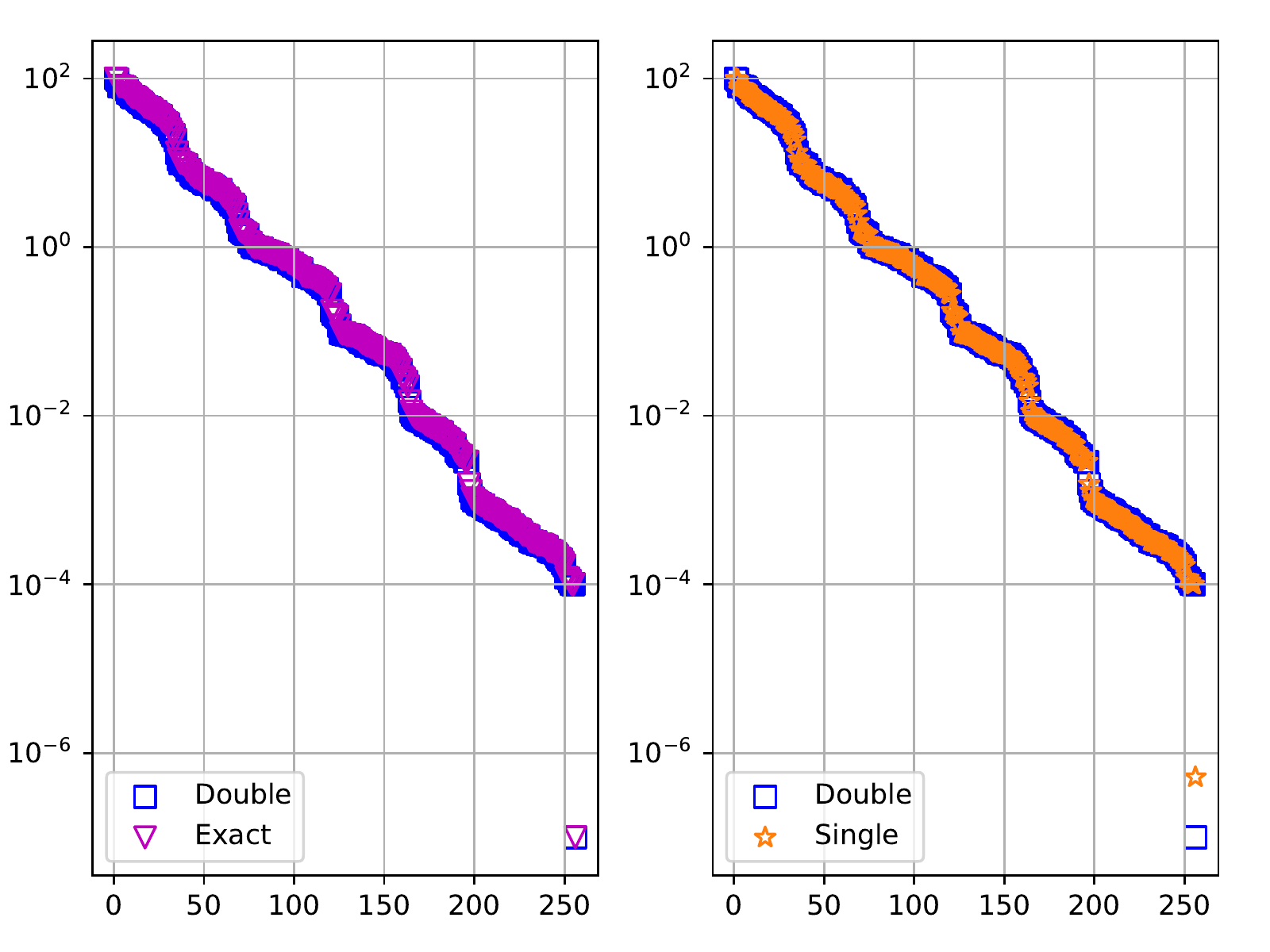}
\caption{
The matrix $\ma\in\real^{4096\times 256}$ has 255 distinct singular values 
in $[10^{-4}, 10^2]$, and a single small singular value $10^{-7}$.
All panels: Double precision singular values (\textcolor{blue}{squares}).
Left: Exact singular values (\textcolor{purple}{triangles}). 
Right: Single precision singular values (\textcolor{orange}{stars}). 
}
\label{fig4.1}    
\end{figure}

\begin{figure}[ht!]
\includegraphics[width=0.85\textwidth]{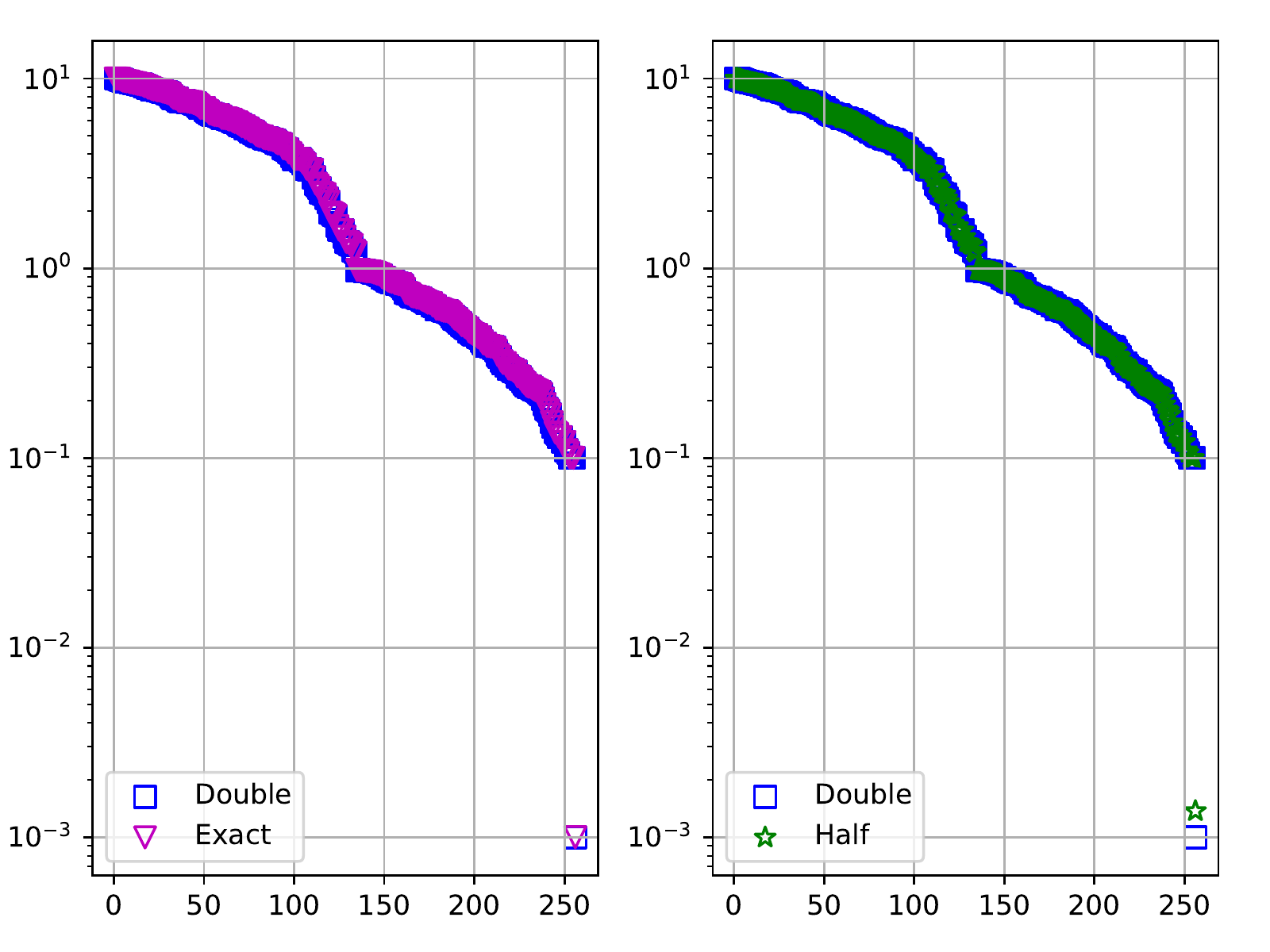}
\caption{
The matrix $\ma\in\real^{4096\times 256}$ has 255 distinct singular values 
in $[10^{-1}, 10^1]$, and a single small singular value $10^{-3}$.
All panels: Double precision singular values (\textcolor{blue}{squares}).
Left: Exact singular values (\textcolor{purple}{triangles}). 
Right: Half precision singular values (\textcolor{green}{stars}).
}\label{fig4.2}    
\end{figure}

\begin{figure}[ht!]
\includegraphics[width=0.85\textwidth]{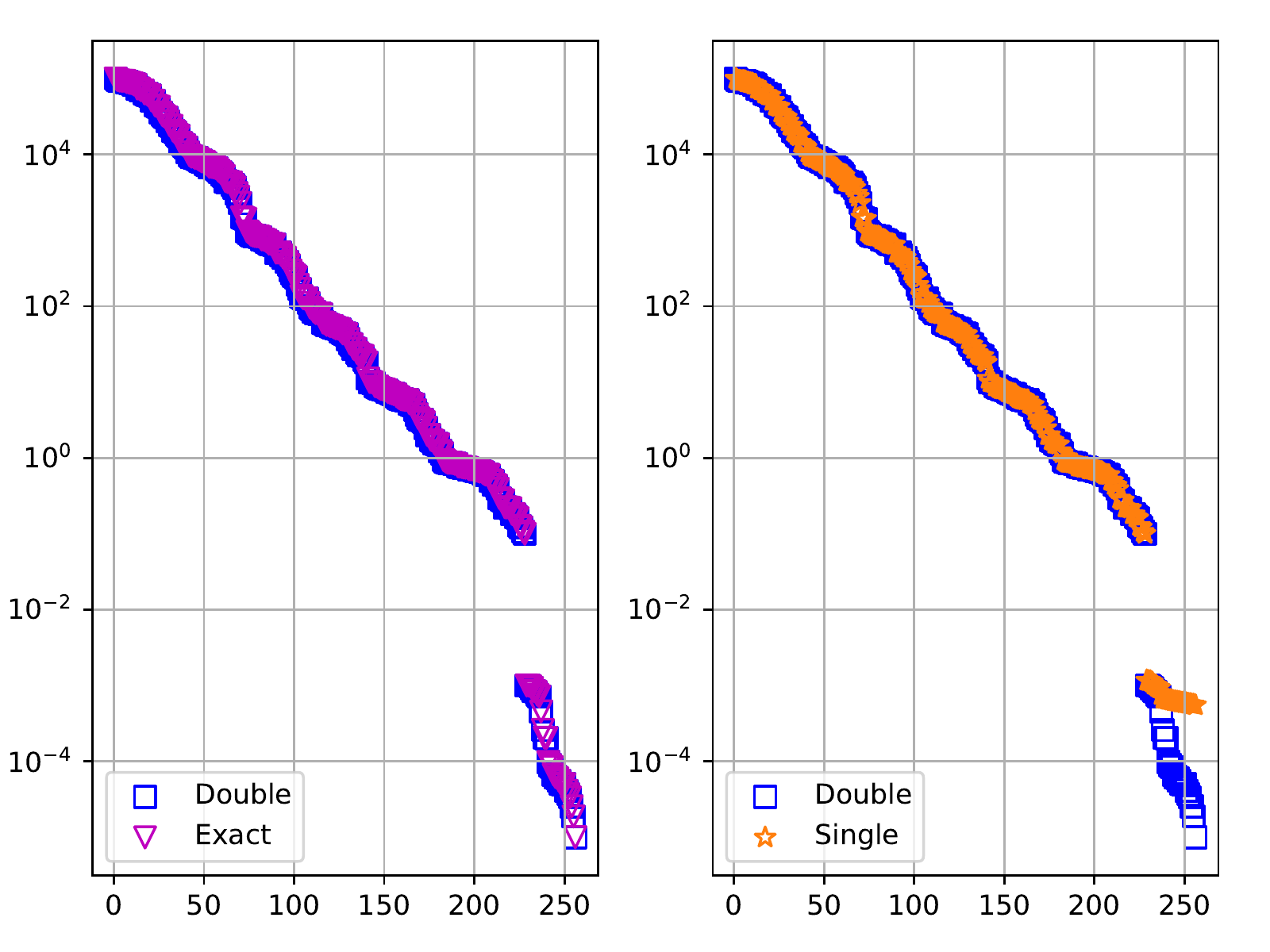}
\caption{The matrix $\ma\in\real^{4096\times 256}$ has 228 distinct singular values 
in $[10^{-1}, 10^5]$, and 28 distinct singular values in $[10^{-5}, 10^{-3}]$.
All panels: Double precision singular values (\textcolor{blue}{squares}).
Left: Exact singular values (\textcolor{purple}{triangles}). 
Right: Single precision singular values (\textcolor{orange}{stars}). 
}\label{fig4.3}    
\end{figure}

\begin{figure}[ht!]
\includegraphics[width=0.85\textwidth]{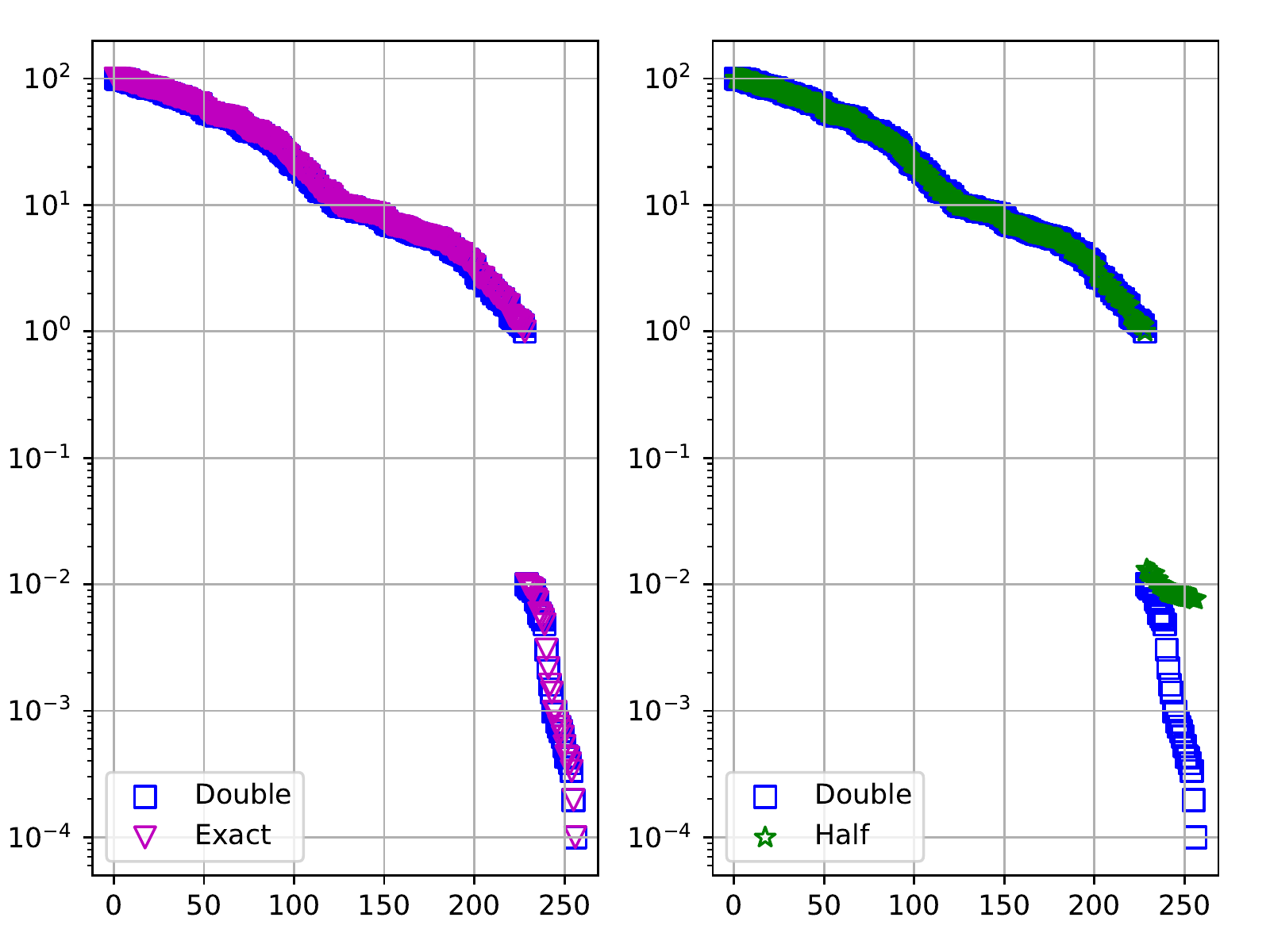}
\caption{
The matrix $\ma\in\real^{4096\times 256}$ has 228 distinct singular values 
in $[10^{0}, 10^2]$, and 28 distinct singular values in $[10^{-4}, 10^{-2}]$.
All panels: Double precision singular values (\textcolor{blue}{squares}).
Left: Exact singular values (\textcolor{purple}{triangles}). 
Right: Half precision singular values (\textcolor{green}{stars}).
}\label{fig4.4}    
\end{figure}

\section{Future Work}\label{s_future}
We investigated the change in the computed 
singular values of a full column-rank
matrix $\ma$ after it has been is downcast
to a lower arithmetic precision.
Our lower bounds 
in Theorem \ref{t_i1} represent a \textit{qualitative} model
for the increase in the 
smallest singular values of the perturbed matrix $\ma+\me$, which 
is confirmed by the experiments in section~\ref{s_exp}.

Future work will consist of a \textit{quantitative} analysis 
to determine the exact order of magnitude of the increase in the small singular 
values and the
structural properties of $\ma$ that can contribute to it, including
specifically 
the size of the gap that separates the small singular values from the larger singular values; and
the condition number of $\ma$ with respect to left inversion.

In addition, the influence of the third order perturbation terms needs to be investigated, as they might possibly become dominant for ill-conditioned matrices~$\ma$.


\appendix 

\section{Algorithms}\label{appen}
We present pseudo codes for two algorithms: The function 
\texttt{create\_sigmas} in Algorithm~\ref{al2} computes the singular values~$\msig$ according to the specifications in the input
parameters \texttt{params}.
Algorithm~\ref{al1} constructs $\ma$ from $\msig$ in double precision, 
and then computes the singular values 
$\msig^d$ of $\ma$, 
$\msig^s$ of $\mathtt{single}(\ma)$, and 
$\msig^h$ of $\mathtt{half}(\ma)$. If $d_1=0$ or $d_2=0$ in Algorithm~\ref{al2}, then the cluster $\msig_{1}$ or $\msig_{2}$ consists of a single singular value of multiplicity $k_{1}$ or $k_{2}$,
respectively.

\begin{algorithm}
\caption{Singular values of $\ma$, $\mathtt{single}(\ma)$ and $\mathtt{half}(\ma)$}
\label{al1}
\begin{algorithmic}
\REQUIRE Large matrix dimension~$m$, \texttt{params}
\ENSURE Singular values of $\ma$ in double, single, half precision 
\STATE{$\msig \gets  \mathtt{create\_sigmas}(\texttt{params})$} 
$\qquad$ \COMMENT{Exact singular values}
\smallskip
\STATE{$n\gets\mathtt{length}(\msig)$}
\qquad \COMMENT{Small dimension of $\ma$}
\smallskip
\STATE{$[\mU,\ms,\mv]\gets \mathtt{SVD}(\mathtt{randn}(m, n))$}
\qquad\COMMENT{Left and right singular vectors for $\ma$}
\smallskip
\STATE{$\ma\gets\mU\msig\mv^T$}
\qquad\COMMENT{Compute $\ma$ in double precision}
\smallskip
\STATE{$ \msig^{d} \gets \mathtt{SVD}(\ma)$}
$\qquad$\COMMENT{Singular values of double precision $\ma$}
\smallskip
\STATE{$ \msig^{s} \gets  \mathtt{SVD}(\mathtt{double}(\mathtt{single}(\ma))$}
$\qquad$\COMMENT{Singular values of single precision $\ma$}
\smallskip
\STATE{$ \msig^{h} \gets \mathtt{SVD}( \mathtt{double}(\mathtt{half}(\ma))$}
$\qquad$\COMMENT{Singular values of half precision $\ma$}
\smallskip
\RETURN $\msig, \, \msig^{d}, \, \msig^{s}, \, \msig^{h}$
\end{algorithmic}
\end{algorithm}

\begin{algorithm}
    \caption{Exact singular values: function \texttt{create\_sigmas}}
    \label{al2}
     \begin{algorithmic}
     \REQUIRE \texttt{params} $=\{s_{1}, g,k_{1},k_{2}, d_{1},d_{2}\}$
      \ENSURE Exact singular values $\msig$
      \smallskip
      \STATE{$\msig \gets \mathtt{zeros}(k_1+k_2,1)$}
      $\qquad$ \COMMENT{Initialize vector of all singular values}
      \smallskip
      \STATE{$\msig_{1} \gets \mathtt{zeros}(k_{1},1)$} $\qquad$\COMMENT{Initialize cluster of large singular values}
       \smallskip
      \STATE{$\msig_2 \gets \mathtt{zeros}(k_{2},1)$} $\qquad$\COMMENT{Initialize cluster of small singular values}
       \smallskip
      \STATE{$\msig_{1}(1) \gets 10^{s_{1}}$} $\qquad$\COMMENT{Largest singular value}
      \smallskip \smallskip
      \IF{$k_{1} > 1$}
        \STATE{$\msig_{1}(k_{1}) \gets 10^{s_{1} - d_{1}}$} $\qquad$\COMMENT{Smallest singular value in $\msig_{1}$}
      \ENDIF
      \smallskip
      \STATE{} \COMMENT{Uniform sampling of interior singular values in cluster $\msig_1$}
      \smallskip
      \FOR{$j=2:k_{1}-1$}{
        \STATE{$\msig_{1}(j) \gets \mathtt{Uniform([\msig_{1}(k_{1}),\msig_{1}(1)])}$} 
      }
      \ENDFOR
      \smallskip
      \STATE{$\msig_{2}(1) \gets 10^{s_{1} - d_{1} - g}$} $\qquad$\COMMENT{Largest singular value in $\msig_2$}
      \smallskip
      \IF{$k_{2} >1$}
        \STATE{$\msig_{2}(k_{2}) \gets 10^{s_{1} - d_{1} - g - d_{2}}$}  $\qquad$\COMMENT{Smallest singular value in $\msig_{2}$}
     \ENDIF
      \smallskip 
      \STATE{} \COMMENT{Uniform sampling of interior singular values  in cluster $\msig_2$}
      \smallskip
      \FOR{$j=2:k_{2}-1$}{
        \STATE{$\msig_{2}(j) \gets \mathtt{Uniform([\msig_{2}(k_{2}),\msig_{2}(1)])}$} 
      }
      \ENDFOR
      \smallskip
      \STATE{$\msig \gets [\msig_{1},\msig_{2}]$} $\qquad$\COMMENT{Concatenate the two singular value clusters}
      \smallskip
      \RETURN $\mathtt{sort}(\msig)$ $\qquad$\COMMENT{Return sorted singular values in non-ascending order}
      \end{algorithmic}
\end{algorithm}

\newpage
\section{Supplementary material} \label{appen_B}
We illustrate that downcasting to lower precision can the increase the set of the smallest singular values, even when the Assumptions~\ref{ass_1}, \ref{ass_2} are not satisfied. We present two additional plots, and specifically Figure~\ref{fig_supp_mat_1} shows the increase of the computed smallest singular values in single ($r=28$) and  Figure~\ref{fig_supp_mat_2} shows the increase of the computed smallest singular value in half ($r=1$). 

\begin{table}[ht!]
\centering
\begin{tabular}{|l|l|l|l|l|l|l|}
\hline
 &$\min(\msig)$ & $\min(\msig^{d})$  & $\min(\msig^{s})$ & $\texttt{avg}(\msig^{s}_2)$  & $\min(\msig^{h})$ & $\texttt{avg}(\msig^{h}_2)$  \\
  \hline
Fig. \ref{fig_supp_mat_1} & $10^{-7}$  & $10^{-7}$  & $5 \times 10^{-5}$ & $4 \times 10^{-5}$  & N/A & N/A  \\
 \hline
 Fig. \ref{fig_supp_mat_2} & $10^{-3}$  & $10^{-3}$  & N/A & N/A  & $6 \times 10^{-3}$  & $5 \times 10^{-3}$  \\
  \hline
\end{tabular}
\caption{Smallest singular values in Figures \ref{fig_supp_mat_1}--\ref{fig_supp_mat_2}: 
exact ($\msig$); double precision ($\msig^d$); single precision ($\msig^s$); 
and half precision ($\msig^h$). The quantities $\texttt{avg}(\msig^{s}_2)$ and $\texttt{avg}(\msig^{h}_2)$ represent the average increase of the $r$ smallest singular values for single and half precision.}
\label{su[[_mat_table}
\end{table}

\paragraph{Figure~\ref{fig_supp_mat_1}}
The cluster $\msig_1$ contains $228$ distinct singular values in the interval $[10^{-3}, 10^4]$, while $\msig_2$ contains the smallest $28$ singular values in the interval $[10^{-7},10^{-4}]$. The values in Assumption~\ref{ass_2} and (\ref{eqn:aspd}) are not satisfied since 
\begin{align*}
    &\underbrace{\sigma_{n-r}}_{10^{-4}}\  \lesssim\  \underbrace{\sigma_{\max}\, u}_{6 \times 10^{-2}}, \\ &\underbrace{4\,\sigma_{\max}\,u}_{2.4 \times 10^{-1}} \ >  
    \underbrace{\sigma_{n-r+1}}_{10^{-3}}.
\end{align*}
However, the smallest singular value of $\msig_{2}$ has increased to $5 \times 10^{-6}$, with an average increase of $4 \times 10^{-6}$ for the $r$ smallest singular values. 

\paragraph{Figure~\ref{fig_supp_mat_2}}
The cluster $\msig_1$ contains 255 distinct singular values
in the interval $[10^{-2}, 10^2]$, while $\msig_2$ contains the single singular value $10^{-3}$.
The values in Assumption~\ref{ass_1} and~(\ref{eqn:aspd}) are not satisfied since
\begin{align*}
    &\underbrace{\sigma_{n}}_{10^{-3}}\  \lesssim\  \underbrace{\sigma_{\max}\, u}_{5 \times 10^{-2}}, \\ &\underbrace{4\, \sigma_{\max}\, u}_{2 \times 10^{-1}} \ > \underbrace{\sigma_{n-1}}_{10^{-2}}.
\end{align*}
However, in half precision, the smallest singular value has increased
to $6\times 10^{-3}$.

\begin{figure}[ht!]
\includegraphics[width=0.85\textwidth]{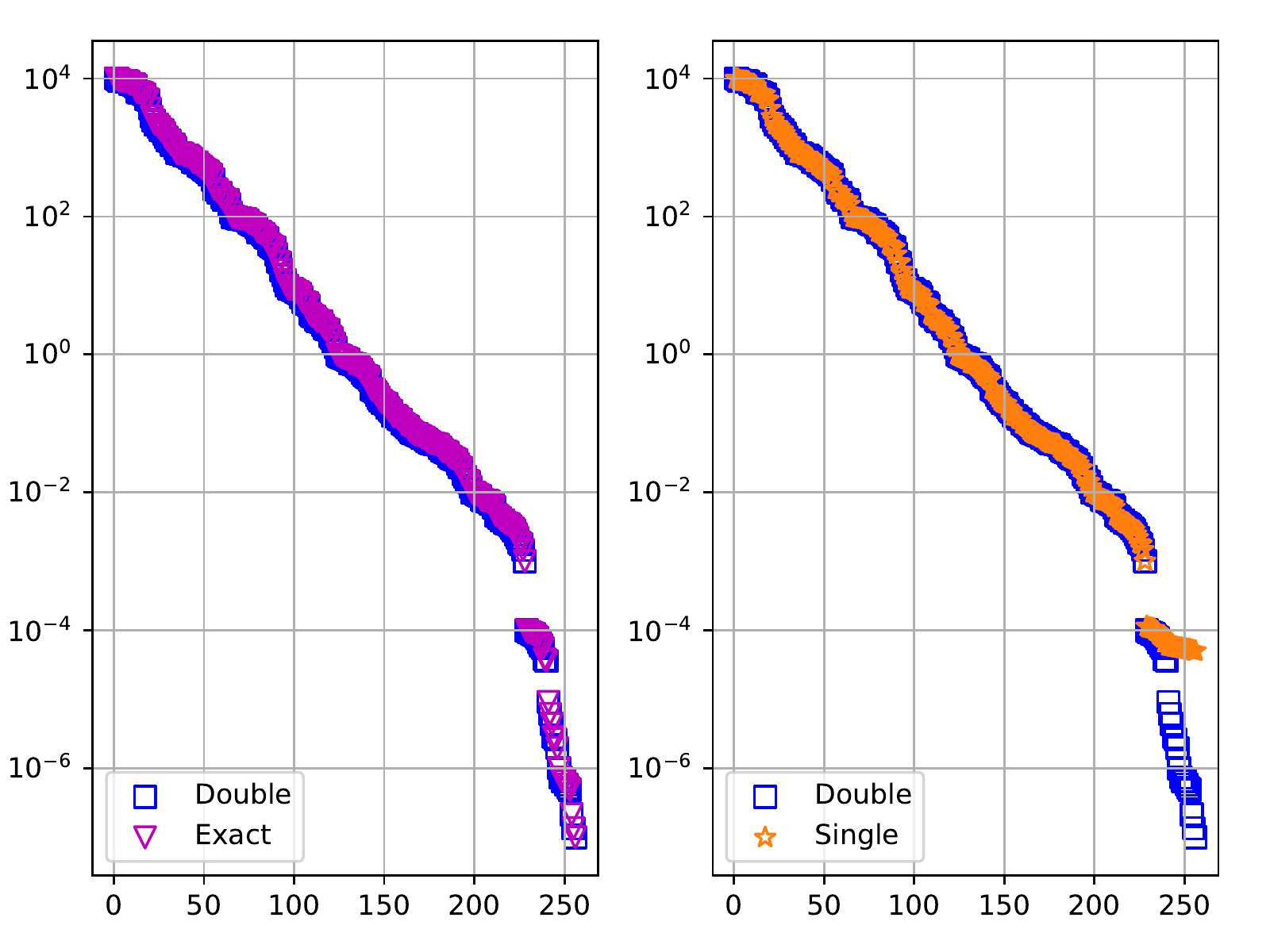}
\caption{The matrix $\ma\in\real^{4096\times 256}$ has 228 distinct singular values 
in $[10^{-3}, 10^4]$, and 28 distinct singular values in $[10^{-7}, 10^{-4}]$.
All panels: Double precision singular values (\textcolor{blue}{squares}).
Left: Exact singular values (\textcolor{purple}{triangles}). 
Right: Single precision singular values (\textcolor{orange}{stars}). 
}\label{fig_supp_mat_1}    
\end{figure}

\begin{figure}[ht!]
\includegraphics[width=0.85\textwidth]{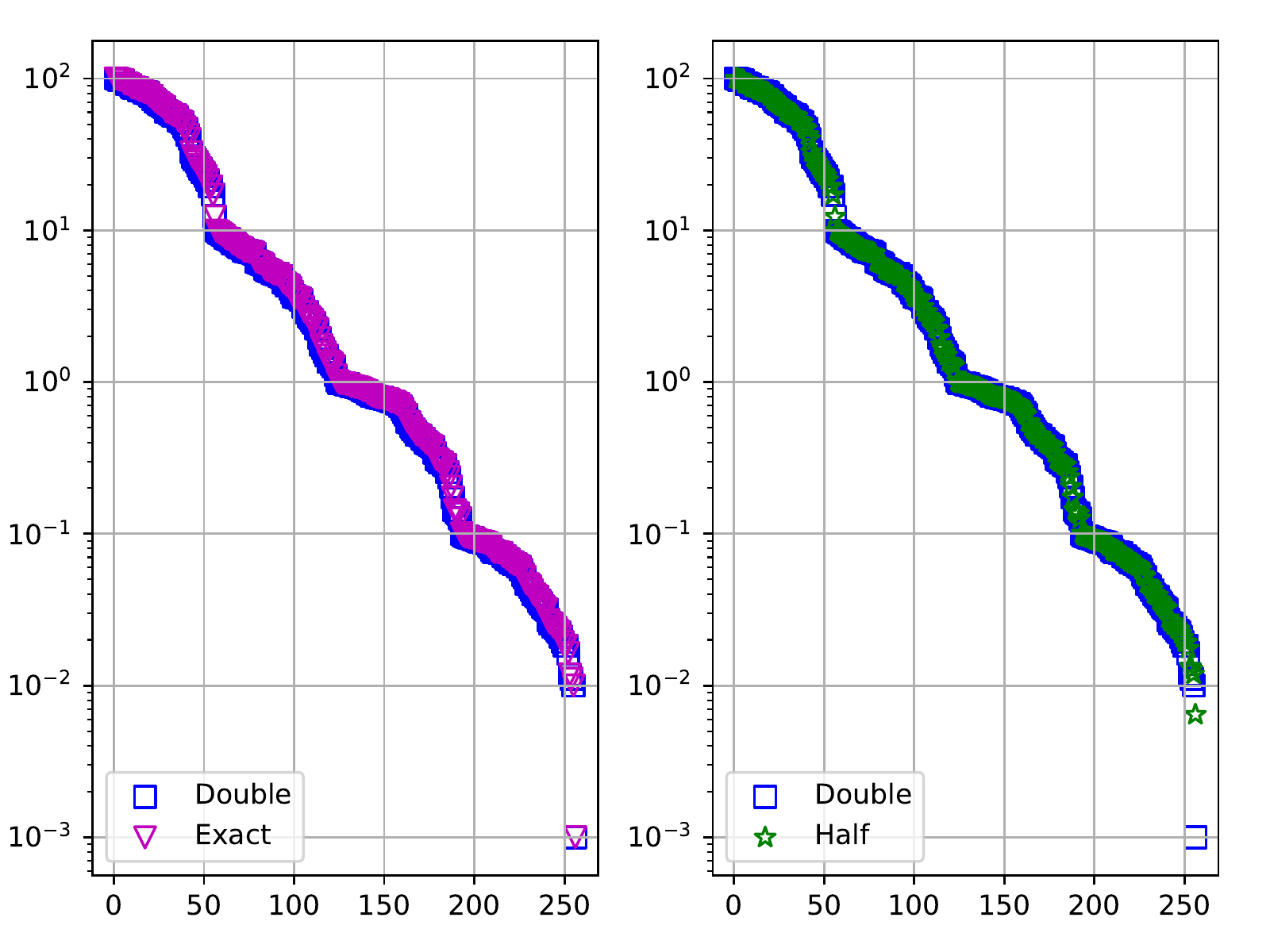}
\caption{The matrix $\ma\in\real^{4096\times 256}$ has 228 distinct singular values 
in $[10^{-2}, 10^2]$, and a single small singular value $10^{-3}$.
All panels: Double precision singular values (\textcolor{blue}{squares}).
Left: Exact singular values (\textcolor{purple}{triangles}). 
Right: Half precision singular values (\textcolor{green}{stars}). 
}\label{fig_supp_mat_2}    
\end{figure}

\printbibliography

@book {GovL13,
    AUTHOR = {Golub, G. H. and {Van Loan}, C. F.},
     TITLE = {Matrix computations},
    SERIES = {Johns Hopkins Studies in the Mathematical Sciences},
   EDITION = {Fourth},
 PUBLISHER = {Johns Hopkins University Press, Baltimore, MD},
      YEAR = {2013}
}

@article {Cook18,
    AUTHOR = {Cook, N.},
     TITLE = {Lower bounds for the smallest singular value of structured
              random matrices},
   JOURNAL = {Ann. Probab.},
     VOLUME = {46},
      YEAR = {2018},
    NUMBER = {6},
     PAGES = {3442--3500},
       DOI = {10.1214/17-AOP1251},
       URL = {https://doi.org/10.1214/17-AOP1251},
}

@article {Demmel88,
    AUTHOR = {Demmel, J. W.},
     TITLE = {The probability that a numerical analysis problem is
              difficult},
   JOURNAL = {Math. Comp.},
     VOLUME = {50},
      YEAR = {1988},
    NUMBER = {182},
     PAGES = {449--480},
          DOI = {10.2307/2008617},
       URL = {https://doi.org/10.2307/2008617},
}

@article {TaoVu2010,
    AUTHOR = {Tao, T. and Vu, V.},
     TITLE = {Smooth analysis of the condition number and the least singular
              value},
   JOURNAL = {Math. Comp.},
    VOLUME = {79},
      YEAR = {2010},
    NUMBER = {272},
       DOI = {10.1090/S0025-5718-2010-02396-8},
       URL = {https://doi.org/10.1090/S0025-5718-2010-02396-8},
}

@incollection {TaoVu2009,
    AUTHOR = {Tao, T. and Vu, V.},
     TITLE = {Smooth analysis of the condition number and the least singular
              value},
 BOOKTITLE = {Approximation, randomization, and combinatorial optimization},
    SERIES = {Lecture Notes in Comput. Sci.},
    VOLUME = {5687},
     PAGES = {714--737},
 PUBLISHER = {Springer, Berlin},
      YEAR = {2009},
  }

@inproceedings {TaoVu2007,
    AUTHOR = {Tao, T. and Vu, V.},
     TITLE = {The condition number of a randomly perturbed matrix},
 BOOKTITLE = {S{TOC}'07---{P}roceedings of the 39th {A}nnual {ACM}
              {S}ymposium on {T}heory of {C}omputing},
     PAGES = {248--255},
 PUBLISHER = {ACM, New York},
      YEAR = {2007},
      DOI = {10.1145/1250790.1250828},
       URL = {https://doi.org/10.1145/1250790.1250828},
}

@article {SST2006,
    AUTHOR = {Sankar, A. and Spielman, D. A. and Teng, S.-H.},
     TITLE = {Smoothed analysis of the condition numbers and growth factors
              of matrices},
   JOURNAL = {SIAM J. Matrix Anal. Appl.},
     VOLUME = {28},
      YEAR = {2006},
    NUMBER = {2},
     PAGES = {446--476},
      ISSN = {0895-4798},
         DOI = {10.1137/S0895479803436202},
       URL = {https://doi.org/10.1137/S0895479803436202},
}

@article {Ste1990,
    AUTHOR = {Stewart, G. W.},
     TITLE = {Stochastic perturbation theory},
   JOURNAL = {SIAM Rev.},
      VOLUME = {32},
      YEAR = {1990},
    NUMBER = {4},
     PAGES = {579--610},
      ISSN = {0036-1445},
}

@article {LotzN2020,
    AUTHOR = {Lotz, M. and Noferini, V.},
TITLE = {Wilkinson's bus: weak condition numbers, with an application
              to singular polynomial eigenproblems},
   JOURNAL = {Found. Comput. Math.},
    VOLUME = {20},
      YEAR = {2020},
    NUMBER = {6},
     PAGES = {1439--1473},
 }

@article {Cucker2016,
    AUTHOR = {Cucker, F.},
     TITLE = {Probabilistic analyses of condition numbers},
   JOURNAL = {Acta Numer.},
     VOLUME = {25},
      YEAR = {2016},
     PAGES = {321--382},
      ISSN = {0962-4929},
}

@article {Arm2010,
    AUTHOR = {Armentano, D.},
TITLE = {Stochastic perturbations and smooth condition numbers},
   JOURNAL = {J. Complexity},
     VOLUME = {26},
 YEAR = {2010},
    NUMBER = {2},
     PAGES = {161--171},
      ISSN = {0885-064X},
 }

@article {KLR1998,
    AUTHOR = {Kenney, C. S. and Laub, A. J. and Reese, M. S.},
     TITLE = {Statistical condition estimation for linear systems},
   JOURNAL = {SIAM J. Sci. Comput.},
     VOLUME = {19},
      YEAR = {1998},
    NUMBER = {2},
     PAGES = {566--583},
      ISSN = {1064-8275},
 }

@misc{Carson22,
  doi = {10.48550/ARXIV.2205.13355},
  url = {https://arxiv.org/abs/2205.13355},
  author = {Carson, E. and Dau\v{z}ickait\.{e}, I.},
  title = {Single-pass {Nystr\"{o}m} approximation in mixed precision},
  note = {arXiv:2202.13355},
  year = {2022}
}

@book {SSun90,
    AUTHOR = {Stewart, G. W. and Sun, J. G.},
     TITLE = {Matrix perturbation theory},
    SERIES = {Computer Science and Scientific Computing},
 PUBLISHER = {Academic Press, Inc., Boston, MA},
      YEAR = {1990}
 }

@article {SteM06,
    AUTHOR = {Stewart, Michael},
     TITLE = {Perturbation of the {SVD} in the presence of small singular
              values},
   JOURNAL = {Linear Algebra Appl.},
    VOLUME = {419},
      YEAR = {2006},
    NUMBER = {1},
     PAGES = {53--77},
      ISSN = {0024-3795},
       DOI = {10.1016/j.laa.2006.04.013}
}

@article {Ste84,
    AUTHOR = {Stewart, G. W.},
TITLE = {A second order perturbation expansion for small singular
              values},
   JOURNAL = {Linear Algebra Appl.},
     VOLUME = {56},
      YEAR = {1984},
     PAGES = {231--235}
}

@book{Par80,
author = "Parlett, B. N.",
title = "The Symmetric Eigenvalue Problem",
publisher = "Prentice Hall",
address = "Englewood Cliffs",  
year = "1980"}

@book {HoJoI,
    AUTHOR = {Horn, R. A. and Johnson, C. R.},
     TITLE = {Matrix analysis},
   EDITION = {Second},
 PUBLISHER = {Cambridge University Press, Cambridge},
      YEAR = {2013}}

@article {Rump09,
    AUTHOR = {Rump, S. M.},
     TITLE = {Inversion of extremely ill-conditioned matrices in
              floating-point},
   JOURNAL = {Japan J. Indust. Appl. Math.},
    VOLUME = {26},
      YEAR = {2009},
    NUMBER = {2-3},
     PAGES = {249--277},
      ISSN = {0916-7005},
       URL = {http://projecteuclid.org/euclid.jjiam/1265033781}
}

@article {Shun22,
    AUTHOR = {Shun, X.},
     TITLE = {Two new lower bounds for the smallest singular value},
   JOURNAL = {J. Math. Inequal.},
     VOLUME = {16},
      YEAR = {2022},
    NUMBER = {1},
     PAGES = {63--68},
        DOI = {10.7153/jmi-2022-16-05},
       URL = {https://doi.org/10.7153/jmi-2022-16-05}
}

@article {Zou12,
    AUTHOR = {Zou, L.},
     TITLE = {A lower bound for the smallest singular value},
   JOURNAL = {J. Math. Inequal.},
    VOLUME = {6},
      YEAR = {2012},
    NUMBER = {4},
     PAGES = {625--629},
      ISSN = {1846-579X},
       DOI = {10.7153/jmi-06-60},
       URL = {https://doi.org/10.7153/jmi-06-60}
}

@article {LinXie21,
    AUTHOR = {Lin, M. and Xie, M.},
     TITLE = {On some lower bounds for smallest singular value of matrices},
   JOURNAL = {Appl. Math. Lett.},
    VOLUME = {121},
      YEAR = {2021},
     PAGES = {Paper No. 107411, 7},
      ISSN = {0893-9659},
       DOI = {10.1016/j.aml.2021.107411},
       URL = {https://doi.org/10.1016/j.aml.2021.107411},
}

@article {Li20,
    AUTHOR = {Li, C.},
     TITLE = {Schur complement-based infinity norm bounds for the inverse of
              {SDD} matrices},
   JOURNAL = {Bull. Malays. Math. Sci. Soc.},
     VOLUME = {43},
      YEAR = {2020},
    NUMBER = {5},
     PAGES = {3829--3845},
      ISSN = {0126-6705},
        DOI = {10.1007/s40840-020-00895-x},
       URL = {https://doi.org/10.1007/s40840-020-00895-x},
}

@article {Sang21,
    AUTHOR = {Sang, C.},
     TITLE = {Schur complement-based infinity norm bounds for the inverse of
              {$DSDD$} matrices},
   JOURNAL = {Bull. Iranian Math. Soc.},
     VOLUME = {47},
      YEAR = {2021},
    NUMBER = {5},
     PAGES = {1379--1398},
      ISSN = {1017-060X},
        DOI = {10.1007/s41980-020-00447-w},
       URL = {https://doi.org/10.1007/s41980-020-00447-w},
}

@article {Huang08,
    AUTHOR = {Huang, T.-Z.},
     TITLE = {Estimation of {$\|A^{-1}\|_\infty$} and the smallest singular
              value},
   JOURNAL = {Comput. Math. Appl.},
    VOLUME = {55},
      YEAR = {2008},
    NUMBER = {6},
     PAGES = {1075--1080},
      ISSN = {0898-1221},
       DOI = {10.1016/j.camwa.2007.04.036},
       URL = {https://doi.org/10.1016/j.camwa.2007.04.036},
}

@article {Varah75,
    AUTHOR = {Varah, J. M.},
     TITLE = {A lower bound for the smallest singular value of a matrix},
   JOURNAL = {Linear Algebra Appl.},
    VOLUME = {11},
      YEAR = {1975},
     PAGES = {3--5},
      ISSN = {0024-3795},
       DOI = {10.1016/0024-3795(75)90112-3},
       URL = {https://doi.org/10.1016/0024-3795(75)90112-3},
}

@article {Johnson89,
    AUTHOR = {Johnson, C. R.},
     TITLE = {A {G}ersgorin-type lower bound for the smallest singular
              value},
   JOURNAL = {Linear Algebra Appl.},
    VOLUME = {112},
      YEAR = {1989},
     PAGES = {1--7},
      ISSN = {0024-3795},
       DOI = {10.1016/0024-3795(89)90583-1},
       URL = {https://doi.org/10.1016/0024-3795(89)90583-1},
}

@article {JohnsonSzulc98,
    AUTHOR = {Johnson, C. R. and Szulc, T.},
     TITLE = {Further lower bounds for the smallest singular value},
   JOURNAL = {Linear Algebra Appl.},
    VOLUME = {272},
      YEAR = {1998},
     PAGES = {169--179},
      ISSN = {0024-3795},
       DOI = {10.1016/S0024-3795(97)00330-3},
       URL = {https://doi.org/10.1016/S0024-3795(97)00330-3},
}

@article {YuGu97,
    AUTHOR = {Yu, Y. and Gu, D.},
     TITLE = {A note on a lower bound for the smallest singular value},
   JOURNAL = {Linear Algebra Appl.},
    VOLUME = {253},
      YEAR = {1997},
     PAGES = {25--38},
      ISSN = {0024-3795},
       DOI = {10.1016/0024-3795(95)00784-9},
       URL = {https://doi.org/10.1016/0024-3795(95)00784-9},
}

@article {HongPan92,
    AUTHOR = {Hong, Y. P. and Pan, C.-T.},
     TITLE = {A lower bound for the smallest singular value},
     JOURNAL = {Linear Algebra Appl.},
    VOLUME = {172},
      YEAR = {1992},
     PAGES = {27--32},
      ISSN = {0024-3795},
       DOI = {10.1016/0024-3795(92)90016-4},
       URL = {https://doi.org/10.1016/0024-3795(92)90016-4},
}

@article {Oishi23,
    AUTHOR = {Oishi, S.},
     TITLE = {Lower bounds for the smallest singular values of generalized asymptotic diagonal
              dominant matrices},
   JOURNAL = {Jpn. J. Ind. Appl. Math.},
    month = {July},
      YEAR = {2023},
DOI = {10.1007/s13160-023-00596-5}
      }

@article{golub1965calculating,
  title={Calculating the singular values and pseudo-inverse of a matrix},
  author={Golub, Gene and Kahan, William},
  journal={Journal of the Society for Industrial and Applied Mathematics, Series B: Numerical Analysis},
  volume={2},
  number={2},
  pages={205--224},
  year={1965},
  publisher={SIAM}
}

\end{document}